\documentclass[11pt]{amsart}
\usepackage[margin=1.2in]{geometry}
\usepackage{amsmath,amssymb,amsthm}
\usepackage[colorlinks=true,linkcolor=blue,citecolor=blue,urlcolor=blue]{hyperref}
\usepackage{graphicx}
\newtheorem{theorem}{Theorem}[section]
\newtheorem{proposition}[theorem]{Proposition}
\newtheorem{lemma}[theorem]{Lemma}
\newtheorem{corollary}[theorem]{Corollary}
\newtheorem{conjecture}[theorem]{Conjecture}
\theoremstyle{definition}
\newtheorem{definition}[theorem]{Definition}
\newtheorem{example}[theorem]{Example}
\newtheorem{question}[theorem]{Question}
\theoremstyle{remark}
\newtheorem{remark}[theorem]{Remark}

\newcommand{\Z}{\mathbb{Z}}
\newcommand{\Q}{\mathbb{Q}}

\newcommand{\R}{\mathbb{R}}
\newcommand{\C}{\mathbb{C}}
\newcommand{\Sph}{S^3}
\newcommand{\lk}{\operatorname{lk}}
\newcommand{\Fix}{\operatorname{Fix}}
\newcommand{\eqdot}{\doteq}
\newcommand{\qint}[1]{[#1]}
\newcommand{\Conway}{\nabla}
\newcommand{\eps}{\varepsilon}
\newcommand{\knotpoly}[1]{p_{#1}}

\usepackage{tikz}
\tikzset{every picture/.append style={line width=0.8pt}}
\newcommand{\wheel}[2][90]{%
\begin{tikzpicture}
  \def\rin{1.0}   
  \def\rout{2.0}  
  \fill[gray!15,even odd rule] (0,0) circle (\rin) (0,0) circle (\rout);
  \draw (0,0) circle (\rin);
  \draw (0,0) circle (\rout);
  \foreach \i in {0,...,\numexpr#2-1\relax}{%
    \pgfmathsetmacro{\spoke}{#1 + \i*360/#2}
    \draw (\spoke:\rin) -- (\spoke:\rout);
    \pgfmathsetmacro{\mid}{\spoke + 180/#2}
    \pgfmathsetmacro{\rot}{\mid}
    \pgfmathparse{int(mod(\i,2))}%
    \ifnum\pgfmathresult=0
      \node[rotate=\rot] at (\mid:{0.5*(\rin+\rout)}) {$\overline{R}$};
    \else
      \node[rotate=\rot] at (\mid:{0.5*(\rin+\rout)}) {$R$};
    \fi
  }
  \fill (0,0) circle (1.6pt);
  \node[anchor=north west,inner sep=1pt] at (2pt,-2pt) {$C_2$};
  \node at (0,{-\rout-0.55}) {$n=#2$};
\end{tikzpicture}}

\colorlet{annbg}{white}

\newcommand{\braidwheel}[3][90]{%
\begin{tikzpicture}
  \def\rin{0.7}
  \def\rout{2.5}
  \xdef\wordlen{0}%
  \foreach \g in {#3}{\pgfmathtruncatemacro{\tmp}{\wordlen+1}\xdef\wordlen{\tmp}}%
  \pgfmathsetmacro{\cellang}{360/#2}%
  \pgfmathsetmacro{\slot}{\cellang/\wordlen}%
  \foreach \i in {0,...,\numexpr#2-1\relax}{%
    \pgfmathsetmacro{\cellstart}{#1 + \i*\cellang}%
    \pgfmathtruncatemacro{\cellsign}{ifthenelse(mod(\i,2)==0,1,-1)}%
    \foreach \g [count=\jj from 0] in {#3}{%
      \pgfmathsetmacro{\ta}{\cellstart + \jj*\slot}%
      \pgfmathtruncatemacro{\geff}{\cellsign*\g}%
      \pgfmathtruncatemacro{\k}{abs(\geff)}%
      \pgfmathsetmacro{\rA}{\rin + 0.25*\k*(\rout-\rin)}%
      \pgfmathsetmacro{\rB}{\rin + 0.25*(\k+1)*(\rout-\rin)}%
      \pgfmathtruncatemacro{\spect}{ifthenelse(\k==1,3,1)}%
      \pgfmathsetmacro{\rS}{\rin + 0.25*\spect*(\rout-\rin)}%
      \draw (\ta:\rS) arc[start angle=\ta, end angle={\ta+\slot}, radius=\rS];
      \ifnum\geff>0
        \draw plot[domain=0:1,samples=21,variable=\t]
          ({\ta+\t*\slot}:{\rB + (\rA-\rB)*0.5*(1-cos(180*\t))});
        \draw[preaction={draw=annbg,line width=3.5pt}]
          plot[domain=0:1,samples=21,variable=\t]
          ({\ta+\t*\slot}:{\rA + (\rB-\rA)*0.5*(1-cos(180*\t))});
      \else
        \draw plot[domain=0:1,samples=21,variable=\t]
          ({\ta+\t*\slot}:{\rA + (\rB-\rA)*0.5*(1-cos(180*\t))});
        \draw[preaction={draw=annbg,line width=3.5pt}]
          plot[domain=0:1,samples=21,variable=\t]
          ({\ta+\t*\slot}:{\rB + (\rA-\rB)*0.5*(1-cos(180*\t))});
      \fi
    }%
  }%
\end{tikzpicture}}

\colorlet{gapbg}{white}

\pgfmathdeclarefunction{hvease}{3}{%
  \begingroup
  \pgfmathparse{ifthenelse(#2==1,
      ifthenelse(#3==1, 0.5*(1-cos(180*#1)), 1-cos(90*#1)),
      ifthenelse(#3==1, sin(90*#1), #1))}%
  \pgfmathsmuggle\pgfmathresult\endgroup}

\newcommand\dpiece[7]{%
  \pgfmathtruncatemacro{\exA}{1-#5}\pgfmathtruncatemacro{\exB}{1-#6}%
  \ifnum#7=1
    \draw[preaction={draw=gapbg,line width=\gw pt}]
      plot[domain=0:1,samples=25,variable=\t]
      ({\slotA + (#2+(#4-#2)*hvease(\t,#5,#6))*\slotW} :
       {\bandR + (#1+(#3-#1)*hvease(\t,\exA,\exB))*\bandW});
  \else
    \draw plot[domain=0:1,samples=25,variable=\t]
      ({\slotA + (#2+(#4-#2)*hvease(\t,#5,#6))*\slotW} :
       {\bandR + (#1+(#3-#1)*hvease(\t,\exA,\exB))*\bandW});
  \fi}

\newcommand\drawcrossing[1]{%
  \pgfmathtruncatemacro{\epse}{#1}
  \pgfmathsetmacro{\pxa}{ifthenelse(\mir==1,1-\rxa,\rxa)}%
  \pgfmathsetmacro{\pxb}{ifthenelse(\mir==1,1-\rxb,\rxb)}%
  \pgfmathsetmacro{\gw}{min(3.5, 9*min((\rxb-\rxa)*\bandW,
      (\ryb-\rya)*\slotW*0.01745*(\bandR+0.5*\bandW)))}%
  \ifnum\epse>0
    \dpiece{\pxb}{\rya}{\pxa}{\ryb}{\cfSE}{\cfNW}{0}
    \dpiece{\pxa}{\rya}{\pxb}{\ryb}{\cfSW}{\cfNE}{1}
  \else
    \dpiece{\pxa}{\rya}{\pxb}{\ryb}{\cfSW}{\cfNE}{0}%
    \dpiece{\pxb}{\rya}{\pxa}{\ryb}{\cfSE}{\cfNW}{1}%
  \fi}
\newcommand\tsig{\drawcrossing{1}}
\newcommand\tsigi{\drawcrossing{-1}}

\newcommand\THii[2]{%
  \pgfmathsetmacro{\xm}{0.5*(\rxa+\rxb)}%
  \begingroup\edef\rxb{\xm}\def\cfNE{1}\def\cfSE{1}#1\endgroup
  \begingroup\edef\rxa{\xm}\def\cfNW{1}\def\cfSW{1}#2\endgroup}
\newcommand\THiii[3]{%
  \pgfmathsetmacro{\xu}{\rxa+(\rxb-\rxa)/3}%
  \pgfmathsetmacro{\xv}{\rxa+2*(\rxb-\rxa)/3}%
  \begingroup\edef\rxb{\xu}\def\cfNE{1}\def\cfSE{1}#1\endgroup
  \begingroup\edef\rxa{\xu}\edef\rxb{\xv}%
    \def\cfNW{1}\def\cfSW{1}\def\cfNE{1}\def\cfSE{1}#2\endgroup
  \begingroup\edef\rxa{\xv}\def\cfNW{1}\def\cfSW{1}#3\endgroup}
\newcommand\TVii[2]{%
  \pgfmathsetmacro{\ym}{0.5*(\rya+\ryb)}%
  \begingroup\edef\ryb{\ym}\def\cfNW{0}\def\cfNE{0}#1\endgroup
  \begingroup\edef\rya{\ym}\def\cfSW{0}\def\cfSE{0}#2\endgroup}

\makeatletter
\def\inscount{\@ifnextchar[{\ins@count}{\ins@count[1]}}
\def\ins@count[#1]#2#3{%
  \pgfmathsetmacro{\tmpw}{\totw+#1}\xdef\totw{\tmpw}}
\def\insdraw{\@ifnextchar[{\ins@draw}{\ins@draw[1]}}
\def\ins@draw[#1]#2#3{%
  \pgfmathsetmacro{\slotA}{\curcellstart + \cumw*\cellang/\totw}%
  \pgfmathsetmacro{\slotW}{#1*\cellang/\totw}%
  \pgfmathtruncatemacro{\keff}{ifthenelse(\mir==1, 3-#2, #2)}%
  \pgfmathsetmacro{\bandR}{\rinX + \keff*(\routX-\rinX)/4}%
  \pgfmathsetmacro{\bandW}{(\routX-\rinX)/4}%
  \foreach \p in {1,2,3}{%
    \ifnum\p=\keff\else\ifnum\p=\numexpr\keff+1\relax\else
      \pgfmathsetmacro{\rS}{\rinX + \p*(\routX-\rinX)/4}%
      \draw (\slotA:\rS) arc[start angle=\slotA,
            end angle={\slotA+\slotW}, radius=\rS];
    \fi\fi}%
  \def\rxa{0}\def\rxb{1}\def\rya{0}\def\ryb{1}%
  \def\cfSW{0}\def\cfSE{0}\def\cfNW{0}\def\cfNE{0}%
  #3%
  \pgfmathsetmacro{\tmpw}{\cumw+#1}\xdef\cumw{\tmpw}}
\makeatother

\newcommand\hvwheel[3][90]{%
\begin{tikzpicture}
  \def\rinX{1.4}\def\routX{4.4}%
  \pgfmathsetmacro{\cellang}{360/#2}%
  \xdef\totw{0}%
  \begingroup\let\ins\inscount#3\endgroup
  \foreach \i in {0,...,\numexpr#2-1\relax}{%
    \pgfmathsetmacro{\curcellstart}{#1 + \i*\cellang}%
    \pgfmathtruncatemacro{\mir}{mod(\i,2)}%
    \xdef\cumw{0}%
    \let\ins\insdraw
    #3%
  }%
\end{tikzpicture}}
\def\JimCell{%
  \ins{1}{\TVii{\tsigi}{\tsigi}}%
  \ins{2}{\THii{\tsig}{\TVii{\tsig}{\tsig}}}%
  \ins[0.55]{1}{\THiii{\tsigi}{\tsigi}{\tsigi}}%
}

\title[A proof of the mod 4 Kawauchi Conjecture]{A proof of the mod 4 Kawauchi Conjecture}

\author{Jim Conant}
\thanks{An initial version of this manuscript was prepared in the course of an extended research conversation between Jim Conant and Claude Fable 5 (an AI assistant by Anthropic), though it has been extensively revised and edited. Every statement, computation, and citation has been verified by the author (with the exception of several computational examples in Appendix B).}

\begin{document}

\begin{abstract}
Kawauchi conjectured that the Conway polynomial of an amphicheiral knot factors as $\Conway_K(z)=f(z)f(-z)$ for some integer polynomial $f(z)$. In joint work with Hartley, he showed this was true for strongly amphicheiral knots, and Hartley used the JSJ decomposition of the knot exterior to generalize to all negative amphicheiral knots. Ermotti--Hongler--Weber were the first to publish a counterexample to the general case. Independently, in 2006 the author had conjectured a statement which is equivalent to the statement that  $\Conway_K \equiv f(z)f(-z) \pmod 4$ for amphicheiral knots, based on certain patterns he noticed in finite type invariants. In this paper we prove this mod 4 version of Kawauchi's original conjecture as a consequence of a stronger integral statement.
\end{abstract}

\maketitle

\section{Introduction}\label{sec:intro}

\subsection{The conjecture}
Recall that a knot $K\subset \Sph$ is \emph{amphicheiral} if it is ambient isotopic to its mirror image $\bar K$, \emph{positive} or \emph{negative} amphicheiral according to whether the isotopy can be chosen to preserve or reverse the orientation of the knot, and \emph{strongly} (positive/negative) amphicheiral if the mirror symmetry is realized by an orientation-reversing involution of the pair $(\Sph,K)$.

Write $\Conway_K(z)\in\Z[z^2]$ for the Conway polynomial, $w=z^2$, and $\Delta_K(t)$ for the Alexander polynomial, related by $\Delta_K(t)\eqdot \Conway_K(z)$ under $z=t^{1/2}-t^{-1/2}$, where $\eqdot$ means up to multiplication by units $t^{\pm k}$.

Hartley and Kawauchi \cite{HK} proved that a strongly positive amphicheiral knot has $\Delta_K(t)=f(t)^2$; equivalently $\Conway_K(z)=\varphi(z^2)^2$. Hartley \cite{H1} proved that every negative amphicheiral knot satisfies
\begin{equation}\label{eq:split}
\Conway_K(z)=f(z)f(-z), \qquad f\in\Z[z].
\end{equation}
Kawauchi \cite{Kaw} conjectured that \eqref{eq:split} holds for all amphicheiral knots; Ermotti, Hongler and Weber \cite{EHW} produced a positive amphicheiral counterexample. On the other hand, in 2006 the author \cite{Con} made a conjecture, arising from the study of finite-type invariants modulo 2, equivalent to the following weakening.

\begin{conjecture}[\cite{Con}]\label{conj:main}
For every amphicheiral knot $K$ there is $f\in\Z[z]$ with
\[
\Conway_K(z)\;\equiv\; f(z)f(-z) \pmod 4 .
\]
\end{conjecture}

Conjecture~\ref{conj:main} was verified in \cite{CM} for periodically amphicheiral knots built from braids where the symmetry preserves the braid structure; that paper also proved a reduction theorem following an argument of Hartley\cite{H1} showing that the general case follows from the case of hyperbolic positive amphicheiral knots, whose orientation-reversing symmetry may be taken to be of finite order by Mostow rigidity.

\subsection{Main results}
The present paper closes the remaining case and, in doing so, proves an integral statement which is strictly stronger than the mod-$4$ congruence and which explains the appearance of the modulus $4$.

\begin{figure}
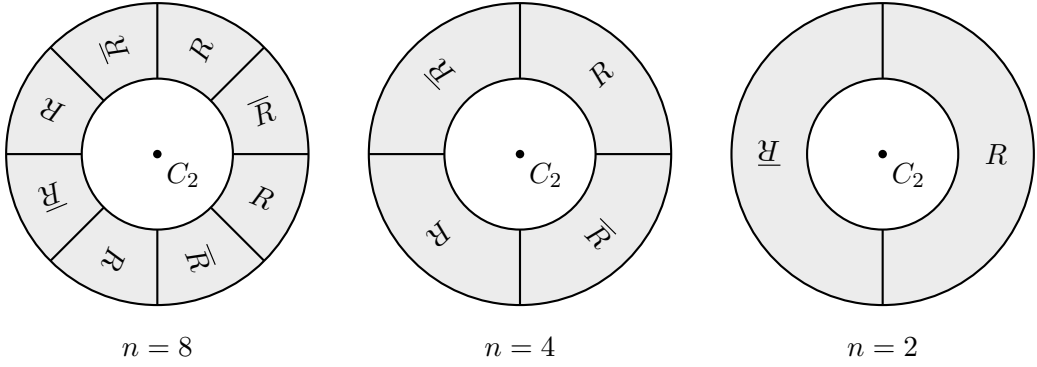

\wheel{8}\qquad \wheel{4}\qquad \wheel{2}
\caption{Template for positive amphicheiral knots invariant under standard rotary reflection of order $n=2^a$. $R$ is a tangle, and $\overline R$ is the reflection of $R$ through the plane of the paper.}\label{fig:stdform}
\end{figure}

Throughout, fix the standard model of an orientation-reversing finite symmetry: for $a\geq 2$ and odd $c$, the \emph{rotary reflection}
\[
\rho\colon \Sph\subset\C^2\to\Sph,\qquad \rho(z_1,z_2)=\big(e^{2\pi i c/2^a}z_1,\ \bar z_2\big),
\]
of order $2^a$, with \emph{reflected axis} $C_2=\{z_1=0\}=\Fix(\rho^2)$. This can be visualized by taking stereographic projection along $\operatorname{Re}(z_2)$, so that the symmetry in $\mathbb R^3$ rotates the first two coordinates around the $z$ axis and reflects the third through the plane of the paper, the $xy$-plane.

In \S\ref{sec:rotary} we show (using geometrization of finite group actions \cite{DL}) that every finite-order orientation-reversing symmetry of a nontrivial knot which preserves the knot's orientation reduces, after passing to an odd power and conjugating, either to this model or to the strongly positive amphicheiral case, and that the assumption $a\geq2$ implies $K\cap C_2=\emptyset$ automatically. 

See Figure~\ref{fig:stdform} for a schematic of knots invariant under rotary reflections. \cite{CM} proved Conjecture~\ref{conj:main} in the case that the tangle $R$ is a braid.

Let $\ell=\lk(K,C_2)$, oriented so $\ell>0$. Then $\ell$ must be odd (Lemma~\ref{lem:elleven}). Define the quantum integer, with its oft-needed sign as
\[
\qint{\ell}\;=\;\frac{t^{\ell/2}-t^{-\ell/2}}{t^{1/2}-t^{-1/2}}\;\in\;\Z[w],
\qquad \eps_\ell=(-1)^{(\ell-1)/2},
\]
so $\qint{\ell}_{z=0}=\ell$. The expression $\eps_\ell[\ell]$ will be referred to as a \emph{signed quantum integer}.

The set up is a \emph{quotient tower}. The involution $r=\rho^{2^{a-1}}$ is the $\pi$-rotation about $C_2$; quotienting by $r$ is a double branched cover of $\Sph$ along the unknot $C_2$ and reproduces the standard configuration with $a$ replaced by $a-1$ (Proposition~\ref{prop:tower}). Iterating produces knots
\[
K=K_a \longrightarrow K_{a-1}\longrightarrow\cdots\longrightarrow K_1,
\]
each invariant under a rotary reflection of order $2^k$ with the same axis and the same linking number $\ell$, the base $K_1$ being strongly positive amphicheiral. Let $L_k=K_k\cup C_2$ with two-variable Alexander polynomial $\Delta_{L_k}(t,y)$, where $t,y$ correspond to the meridians of $K_k$ and $C_2$.

\begin{theorem}[Palindromic normal form; Theorem~\ref{thm:normalform}]\label{thm:A}
For $1\le k\le a-1$ there is $Q_k\in\Z[u,v]$ with
\[
\Delta_{L_k}(t,y)\;\eqdot\; Q_k\big(t+t^{-1},\,y+y^{-1}\big),
\]
normalized so that $Q_k(w+2,\,2)=\qint{\ell}\,\Conway_{K_k}(w)$ \textup{(}note $t+t^{-1}=w+2$\textup{)}.
\end{theorem}

The $y$-palindromicity here is \emph{not} a formal property of link polynomials: it is the composite of Torres duality $\Delta(t^{-1},y^{-1})\eqdot\Delta(t,y)$ with the semilinear symmetry $\Delta(t^{-1},y)\eqdot\Delta(t,y)$ induced by the rotary reflection itself, which reverses the ambient orientation and the axis while preserving the knot. However, this palindromicity allows for a cute fact modulo $4$. Standard covering space analysis applied to the double cover around the axis $C_2$  implies that $\Delta_{L_k}(t,y^2)=\Delta_{L_{k-1}}(t,y)\Delta_{L_{k-1}}(t,-y)$. Substituting $y=1$ recovers the knot up to a factor of $[\ell]$, and $(y^p+y^{-p})|_{y=1}\equiv (y^p+y^{-p})_{y=-1}\mod 4$, so that we have $$\Conway_{K_k}(z) \equiv \varepsilon_\ell[\ell] (\Conway_{K_{k-1}}(z))^2 \mod 4.$$

This expression factors mod 4 as $f(z)f(-z)$ due to a cyclotomic identity of potential independent interest (Lemma~\ref{lem:qint}): for every odd $\ell$ there is $B\in\Z[\xi]$, $\xi^2=w+4$, with
\[
B(\xi)\,B(-\xi)\;=\;\eps_\ell\,\qint{\ell}
\]
\emph{exactly over $\Z$}; reducing mod $4$, where $\xi^2\equiv z^2$, this exhibits $\eps_\ell\qint{\ell}$ as $f(z)f(-z)$. This is the second place in which the modulus $4$ is needed: 
modulo $4$, the two quadratic twists $\Z[z]\subset\Z[s^{\pm1}]\supset\Z[\xi]$ (where $z=s-s^{-1}$, $\xi=s+s^{-1}$, $s=t^{1/2}$) become indistinguishable: $\xi^2=z^2+4\equiv z^2$.
This neatly establishes Conjecture~\ref{conj:main} for knots invariant under $2^a$-rotary reflections, but we can do better:

\begin{theorem}[Corollary~\ref{cor:order4}]\label{cor:intro-order4}
If $K$ is invariant under an orientation-reversing symmetry of order $2^a$ ($a\geq 2$) preserving the orientation of $K$, then there are $M,E\in\Z[w]$ with $M(0)$ odd and
\[
\varepsilon_\ell\qint{\ell}\,\Conway_K \;=\; M(w)^2 - 4\,E(w)^2
\]
exactly over $\Z$. 
\end{theorem}

By analyzing the Alexander polynomial of a general positive amphicheiral knot using the JSJ decomposition, we are able to prove the following restriction on the Conway polynomials of general positive amphicheiral knots over the integers.

\begin{theorem}[Theorem~\ref{thm:conjecture}]\label{thm:D}
For any positive amphicheiral knot, we have an exact integer identity 
\[\left(\prod_i \eps_{\ell_i}[m_i\ell_i]\right)\Conway_K=\left(\prod_i[m_i]\right)\left(M^2-4E^2\right),\] for some odd positive integers $\ell_i,m_i$.  
\end{theorem}

A little algebra gives the mod $4$ Kawauchi conjecture as a corollary.
\begin{corollary}
The Conway polynomial of every amphicheiral knot is congruent to $f(z)f(-z)$ modulo $4$.
\end{corollary}

\subsection{Examples}\label{subsec:evidence}
Conway polynomials for an infinite class of braid examples with order 4 symmetry are computed in \cite{CM}, while the Ermotti--Hongler--Weber knot forms a non-braid order $4$ example.
Both can be tested against Theorem~\ref{cor:intro-order4}. The \cite{CM} examples have
\[
\Conway_K = X\big(4-X(w+3)\big), \qquad X=F_n^2F_m^2,
\]
($F_j$ = Fibonacci polynomials, $n,m$ odd), where we notice $[3]=w+3$. Let $T=\qint3 X$. We have $-[3]\Conway_K=T(4-T)=(T-2)^2-4$, which we recognize as an instance of Theorem~\ref{cor:intro-order4} with $E=1$. We reprove this result in the appendix, as well as giving some $5$ strand braid computations where $E\neq 1$ becomes possible. The Ermotti--Hongler--Weber knot, with $\ell=3$, also satisfies Theorem~\ref{cor:intro-order4} on the nose, with explicit $M,E$ recorded in \S\ref{sec:examples}.

\subsection{More detail on the $2^a$ case}
Although
we already have enough to prove the mod $4$
Kawauchi conjecture and its integral strengthening, the Conway polynomials of order $2^a$ periodically positive amphicheiral knots have a lot more structure, which we hint at in the following three theorems. Recall that $Q_k$ is the palindromic Conway-like normalization of the 2-variable Alexander polynomial of the link $L_k$.

\begin{theorem}[Theorem~\ref{thm:telescope}]\label{thm:B}
Set $\widehat R_k(w) := \eps_\ell\, Q_k(w+2,\,-2)\in\Z[w]$. Then $\widehat R_k(0)=1$, and
\[
\Conway_{K_{k+1}} \;=\; \Conway_{K_k}\cdot \widehat R_k \quad (1\le k\le a-1),
\qquad\text{hence}\qquad
 \Conway_K \;=\; \varphi(w)^2\,\prod_{k=1}^{a-1}\widehat R_k(w)
\]
exactly in $\Z[w]$, where $\Conway_{K_1}=\varphi^2$ is the Hartley--Kawauchi square of the base of the tower.
\end{theorem}

\begin{theorem}[Mod 4; Theorem~\ref{thm:mod4}]\label{thm:C}
$\widehat R_k \equiv \eps_\ell\qint{\ell}\,\Conway_{K_k} \pmod 4$ for each $k$, and consequently
\[
\Conway_K \;\equiv\; \big(\eps_\ell\qint{\ell}\big)^{2^{a-1}-1}\,\varphi^{2^{a}}\pmod 4 .
\] 
\end{theorem}
This is yet another way to prove $\Conway_K \equiv f(z)f(-z)\pmod 4$ for knots invariant under the standard rotary reflection of order $2^a$, 
again using that $\varepsilon_\ell[\ell]=f(z)f(-z)\mod 4$. 

Finally, we present a beautiful decomposition of the Conway polynomial in terms of roots of unity and the polynomial $Q_1$ at the base of the tower. Here is the simplest statement, but compare Theorem~\ref{thm:cyclores} for a further elaboration.

\begin{theorem}
Let $K$ be a $2^a$ periodically positive amphicheiral knot in the standard rotary model,
$a\geq2$, with axis linking number $\ell$, and let $Q_1(u,v)$ be the palindromic normal form
\textup{(}Theorem~\textup{\ref{thm:normalform}}\textup{)} of the Alexander polynomial of the first quotient link
$L_1$. Then
\begin{equation}
\eps_\ell\qint\ell\,\Conway_K
\;=\;\prod_{\zeta^{2^{a-1}}=1} Q_1\big(u,\ \zeta+\zeta^{-1}\big)
\end{equation}
where the product runs over all $2^{a-1}$-st roots of unity.
\end{theorem}

\subsection*{Acknowledgments} This paper was developed in an extended research dialogue with Claude (Anthropic; model Claude Fable 5, July 2026), which proposed the quotient-tower strategy, the level-wise torsion identities, and drafted an initial version of the manuscript, which has since been heavily revised. All results produced by Claude, with the exception of specific example calculations in Appendix B, were independently verified by the author, who bears sole responsibility for the correctness of all results in this paper. 

\section{Split polynomials modulo 4}\label{sec:S}

Fix $w=z^2$. 

\begin{definition}\label{def:S}
Let $S\subset \Z_4[w]$ be the set of reductions mod $4$ of polynomials $f(z)f(-z)$, $f\in\Z[z]$.
\end{definition}

Conjecture~\ref{conj:main} states that the Conway polynomial of an amphicheiral knot lies in $S$. 

\begin{lemma}\label{lem:normform}
$S=\{\,g(w)^2-w\,h(w)^2 \bmod 4 \;:\; g,h\in\Z[w]\,\}$, and $S$ is closed under multiplication. In particular $S$ contains all squares $g(w)^2$. 
\end{lemma}

\begin{proof}
Writing $f(z)=g(z^2)+zh(z^2)$ gives $f(z)f(-z)=g(w)^2-wh(w)^2$, and conversely. Multiplicative closure follows directly from Definition~\ref{def:S}.
\end{proof}

This is a ``norm form'' for members of $S$. Write $\Z[z]\cong \Z[w]\oplus z\Z[w]$, whereupon multiplication by an element of $\Z[z]$ is a $2\times2$ matrix, and the determinants of such matrices modulo 4 form the class $S$. 

It turns out that testing membership in $S$ is quite straightforward: there is a unique mod 2 witness.

\begin{lemma}\label{lem:rigidity}
There is a bijection $\Z_2[z]\to S\subseteq \Z_4[w]$ given by $f\mapsto \tilde f(z)\tilde f(-z)$ for any lift $\tilde f\in \Z_4[z].$
\end{lemma}
\begin{proof}
This map is well defined: $f+2g\mapsto f(z)f(-z)+2\big(f(z)g(-z)+f(-z)g(z)\big)$, and $f(z)g(-z)+f(-z)g(z)=h(z)+h(-z)$ is clearly divisible by $2$. Moreover, given $s=f(z)f(-z)\in S$, modulo 2 $f$ is a square root of $s$. Squaring is injective (and linear) in the ring $\Z_2[z]$, so $f \pmod 2$ is uniquely determined.
\end{proof}

For $p\in\mathbb{Z}_4[z]$ let
$\bar p\in\mathbb{F}_2[z]$ denote its reduction modulo $2$.
We call $\bar f\in \Z_2[z]$ the \emph{witness} to the element $f(z)f(-z)\in S$.

\begin{lemma}\label{lem:reduce}
Let $s\in S$ with at least one odd coefficient. Suppose $sf(z)\in S$. Then $f(z)\in S$.
\end{lemma}
\begin{proof}
Note $\bar s\neq 0$ since $s$ has an odd coefficient.

First we show $s$ is not a zero divisor.
Suppose $sX=0$. If $\bar X\neq 0$, then $\bar s\,\bar X\neq 0$ because
$\mathbb{F}_2[z]$ is an integral domain, contradicting $sX=0$. So $\bar X=0$;
write $X=2Y$. Then $0=sX=2sY$ forces $sY\equiv 0\pmod 2$, i.e.\
$\bar s\,\bar Y=0$, whence $\bar Y=0$ and $X=2Y=0$ in $\mathbb{Z}_4[z]$.

Let $\bar g$ be the witness to $s$ and $\bar h$ the witness to $sf$, with $g$ and $h$ mod $4$ lifts. Then since these are mod 2 square roots, we have $\bar g^2\bar f=\bar h^2$. $\Z_2[z]$ is a UFD, so $\bar g\,|\,\bar h$. Set $\bar c=\bar h/\bar g$ so that $\bar f=\bar c^2$.

Now choose any lift $c\in\mathbb{Z}_4[z]$ of $\bar c$. Then $gc\equiv h\pmod 2$, so
Lemma~\ref{lem:rigidity} gives
\[
sf \;=\; h(z)h(-z) \;=\; (gc)(z)\,(gc)(-z)
   \;=\; g(z)g(-z)\cdot c(z)c(-z) \;=\; s\cdot c(z)c(-z).
\]
Therefore $s\bigl(f-c(z)c(-z)\bigr)=0$, and since $s$ is not a zero divisor, we conclude
$f=c(z)c(-z)\in S$.
\end{proof}

The following identity is the arithmetic heart of the mod-$4$ phenomenon. Let $s=t^{1/2}$, and inside $\Z[s^{\pm1}]$ consider
\[
z=s-s^{-1},\qquad \xi=s+s^{-1},\qquad \xi^2=z^2+4=w+4 .
\]
The involution $\sigma\colon s\mapsto -s^{-1}$ fixes $z$ and negates $\xi$; the involution $s\mapsto s^{-1}$ fixes $\xi$ and negates $z$. The fixed ring of $s\mapsto s^{-1}$ is $\Z[\xi]$ and the fixed ring of $\sigma$ is $\Z[z]$ (both standard: $\Z[s^{\pm1}]$ is free of rank 2 over either subring, with $s^2=\xi s-1=zs+1$).

\begin{lemma}[Signed quantum integers split]\label{lem:qint}
Let $\ell\ge1$ be odd and $\qint{\ell}=(s^\ell-s^{-\ell})/(s-s^{-1})\in\Z[w]$. Define
\[
B \;:=\; s^{-(\ell-1)/2}\,\frac{s^{\ell}-1}{s-1}\;=\;s^{-(\ell-1)/2}\big(1+s+\cdots+s^{\ell-1}\big).
\]
Then $B\in\Z[\xi]$, of degree $(\ell-1)/2$ in $\xi$, and
\[
B(\xi)\,B(-\xi)\;=\;\eps_\ell\,\qint{\ell},\qquad \eps_\ell=(-1)^{(\ell-1)/2},
\]
exactly in $\Z[w]$. Consequently $\eps_\ell\qint{\ell}\bmod 4\in S$.
\end{lemma}

\begin{proof}
This was proven in \cite{CM}, Lemma 3.8, by referring to properties of Lucas polynomials. We reprove it from scratch here.

$B$ is a Laurent polynomial in $s$ invariant under $s\mapsto s^{-1}$, hence lies in $\Z[\xi]$. Since $\sigma(\xi)=-\xi$ and $\sigma$ fixes $\Z[w]$, we compute $B(-\xi)=\sigma(B)$. With $\ell$ odd,
\[
\sigma(B)=(-s^{-1})^{-(\ell-1)/2}\,\frac{(-s^{-1})^{\ell}-1}{-s^{-1}-1}
=\eps_\ell\, s^{(\ell-1)/2}\cdot\frac{-s^{-\ell}-1}{-s^{-1}-1}
=\eps_\ell\, s^{-(\ell-1)/2}\,\frac{s^{\ell}+1}{s+1}.
\]
Hence
\[
B\,\sigma(B)=\eps_\ell\, s^{-(\ell-1)}\,\frac{(s^{\ell}-1)(s^{\ell}+1)}{(s-1)(s+1)}
=\eps_\ell\,\frac{s^{\ell}-s^{-\ell}}{s-s^{-1}}=\eps_\ell\qint{\ell}.
\]
For the last claim write $B=g(\xi^2)+\xi h(\xi^2)$; then
\[
\eps_\ell\qint{\ell}=B(\xi)B(-\xi)=g(w+4)^2-(w+4)\,h(w+4)^2\equiv g(w)^2-w\,h(w)^2\pmod 4,
\]
using $(w+4)^k\equiv w^k\pmod4$. By Lemma~\ref{lem:normform}, $\eps_\ell\qint{\ell}\bmod4\in S$.
\end{proof}

\begin{remark}
The identity of Lemma~\ref{lem:qint} says that the signed quantum integer $\eps_\ell[\ell]$ is a norm from the ``wrong'' quadratic twist $\Z[\xi]$; the two twists $\Z[z]$ and $\Z[\xi]$ are the same modulo~$4$ because $\xi^2-z^2=4$. 
\end{remark}

\begin{example}\
Tracing through the definitions, $\qint{\ell}=s^{\ell-1}+s^{\ell-3}+\cdots+s^{-(\ell-1)}$.
As observed in \cite{Agle}, the symmetric powers can be expressed
$s^{2k}+s^{-2k}=L_{2k}(z)$ where $L_{2k}(z)$ is a Lucas polynomial of order $k$.
Hence $\qint{\ell}=L_{\ell-1}+L_{\ell-3}+\cdots+L_2+L_0$.
We have
\begin{align*}
L_0(z)&=1, & L_2(z)&=z^2+2, & L_4(z)&=z^4+4z^2+2,\\
L_6(z)&=z^6+6z^4+9z^2+2, & L_8(z)&=z^8+8z^6+20z^4+16z^2+2.
\end{align*}
This allows us to compute
\begin{itemize}
\item $-\qint3=-w-3=-(\xi^2-4)-3=(\xi+1)(-\xi+1)$
\item $\qint5=w^2+5w+5=\xi^4-3\xi^2+1=(\xi^2+\xi-1)(\xi^2-\xi-1)$
\item $-\qint7=-w^3-7w^2-14w-7=-\xi^6+5\xi^4-6\xi^2+1
       =(\xi^3+\xi^2-2\xi-1)(-\xi^3+\xi^2+2\xi-1)$
\item $\qint9=w^4+9w^3+27w^2+30w+9
       =(\xi^4+\xi^3-3\xi^2-2\xi+1)(\xi^4-\xi^3-3\xi^2+2\xi+1)$
\end{itemize}
\end{example}

\section{Conventions on link polynomials}\label{sec:conventions}

For an ordered oriented $2$-component link $L=K\cup A\subset\Sph$ with exterior $V$, we write $\Delta_L(t,y)\in\Z[t^{\pm1},y^{\pm1}]$ for the two-variable Alexander polynomial, where $t$ and $y$ are the classes of the oriented meridians of $K$ and $A$ in $H_1(V)\cong\Z^2$. It is well defined up to units $\pm t^\alpha y^\beta$; we write $\eqdot$ for equality up to such units. We use two classical facts \cite{Tor}:
\begin{align}
&\text{(Torres specialization)}\qquad \Delta_L(t,1)\;\eqdot\;\qint{\ell}\;\Conway_K(z), \qquad \ell=\lk(K,A),\label{eq:torres}\\
&\text{(Torres duality)}\qquad\quad\ \ \Delta_L(t^{-1},y^{-1})\;\eqdot\;\Delta_L(t,y).\label{eq:duality}
\end{align}
(Here \eqref{eq:torres} is the usual $\Delta_L(t,1)\eqdot\frac{t^\ell-1}{t-1}\Delta_K(t)$ rewritten via $\frac{t^\ell-1}{t-1}=t^{(\ell-1)/2}\qint{\ell}$.) 

We note that the $1$-variable Alexander polynomial has an elementary
interpretation as the order of the \emph{Alexander module}, the first homology
of the infinite cyclic cover of the knot complement viewed as a
$\Q[t^{\pm1}]$-module. This is a finitely generated torsion module over a PID,
so it decomposes as a direct sum of cyclic modules
$\bigoplus_i \Q[t^{\pm1}]/(a_i(t))$, and $\Delta_K(t)\eqdot\prod_i a_i(t)$
(over $\Q$ the unit ambiguity also admits nonzero rational multiples; the
integral normalization comes from a presentation matrix over $\Z[t^{\pm1}]$).
For links the corresponding module is the first homology of the universal
abelian cover $\tilde V$, but the coefficient ring is no longer a PID and the
module need not decompose into cyclics; classically one extracts $\Delta_L$
instead as the greatest common divisor of the maximal minors of a presentation
matrix. We prefer the torsion point of view, which packages the same
information more functorially: the cellular chain complex $C_*(\tilde V)$ is a
finite based complex of free $\Z[t^{\pm},y^{\pm}]$-modules, the \emph{Alexander
complex} of $V$, and for $2$-component links (unlike knots, where a factor
$t-1$ intervenes) its Reidemeister torsion satisfies
$\tau(V)\eqdot\Delta_L(t,y)$ over the field of fractions of
$\Z[t^{\pm},y^{\pm}]$, whenever $\Delta_L\neq0$ \cite{Mil,Tur}. In particular
a homeomorphism of exteriors carries one Alexander complex to the other, \emph{semilinearly} over the induced
map $\theta$ on $H_1$ (i.e. additively, with scalars twisted through $\theta$:
$\tilde F_*(\lambda\, c)=\theta(\lambda)\,\tilde F_*(c)$), which is the source
of the symmetry \eqref{eq:functoriality} below. In fact, Milnor reproves the Torres
duality result above as a special case of the duality theorem for Reidemeister
torsion applied to the Alexander chain complex, which is the main point of his
paper.

If $F\colon V\to V$ is a homeomorphism inducing $\theta$ on $H_1(V)$, then functoriality of the Alexander module $H_1(\tilde V)$ gives the semilinear symmetry
\begin{equation}\label{eq:functoriality}
\Delta_L\big(\theta(t),\theta(y)\big)\;\eqdot\;\Delta_L(t,y).
\end{equation}

Note the identity of subrings of $\Z[t^{\pm1}]$:
\begin{equation}\label{eq:uw}
u:=t+t^{-1}=w+2,\qquad\text{so } \Z[u]=\Z[w].
\end{equation}

\section{Rotary reflections}\label{sec:rotary}
We begin this section by proving that a positive amphicheiral knot realized by a finite order amphicheiral symmetry can be put in standard rotary reflection form in such a way that the knot avoids the rotation axis. This allows us to do computations with the $2$ component link $K\cup C_2$.

In fact, we don't quite prove this, leaving open the question of whether a strongly positive amphicheiral knot can intersect the rotation axis. Luckily we won't need this fact anyway.
\begin{lemma}\label{lem:taxonomy}
Let $K\subset\Sph$ be a nontrivial knot invariant under an orientation-reversing diffeomorphism $h$ of finite order $n$ with $h|_K$ orientation-preserving. Write $n=2^aq$, $q$ odd, $a\ge1$, and set $g:=h^q$, an orientation-reversing diffeomorphism of order $2^a$ with $g|_K$ orientation-preserving. Then:
\begin{enumerate}
\item $g|_K$ is free \textup{(}acts freely on $K$\textup{)};
\item if $a=1$, then $\Fix(g)=S^0$ and $K$ is strongly positive amphicheiral;
\item If $a\geq2$, then $g$ is conjugate in $\mathrm{Diff}(\Sph)$ to the standard rotary reflection $\rho(z_1,z_2)=(e^{2\pi ic/2^a}z_1,\bar z_2)$ for some odd $c$, and $K\cap C_2=\emptyset$ where $C_2=\Fix(g^2)$.
\end{enumerate}
\end{lemma}

\begin{proof}
By \cite{DL}, the finite cyclic group $\langle g\rangle$ is conjugate to a subgroup of $O(4)$; assume $g\in O(4)$. An orientation-reversing orthogonal transformation of $\R^4$ of order $2^a$ has, in a suitable orthonormal basis, the form $R(\theta)\oplus\mathrm{diag}(1,-1)$ with $\theta=2\pi c/2^a$. If $a=1$, we could have $c=0$, which would be a reflection through a plane, but that would force $K$ to be planar, hence trivial. So $c$ must be nonzero, and 
the order condition forces $c$ odd. This is a standard rotary reflection, with $\Fix(g)=\{\pm e_3\}\cong S^0\subset C_2$.

(1) Suppose $g^j(x)=x$ for some $x\in K$, $0<j<2^a$. Then $g^j|_K$ is a finite-order orientation-preserving homeomorphism of a circle with a fixed point, hence $g^j|_K=\mathrm{id}_K$ and $K\subset\Fix(g^j)$. Given that $g$ is (conjugate to) a standard rotary reflection, the possible fixed point sets of its iterates are $S^0$ for odd powers and $C_2$ for even powers. This would mean $K$ is a subset of one of those two sets, which is impossible for $S^0$ and would imply $K$ is unknotted for $C_2$, also impossible. Thus $g|_K$ is free.

(2) We already ruled out a pure reflection for $g$, hence $g$ is a standard rotary reflection of order $2$. A strongly positive amphicheiral knot is preserved by an orientation reversing involution that preserves strand direction, which describes $g$.

(3) We've already established that $g$ can be taken to be a standard rotary reflection. The fact that $K\cap C_2=\emptyset$ follows from (1), since a point in the intersection would be a fixed point of $g^2\neq\mathrm{id}$. 
\end{proof}

We henceforth fix the standard model with its two invariant circles: the \emph{rotated core} $C_1=\{z_2=0\}$ and the \emph{reflected axis} $C_2=\{z_1=0\}$. Note $g|_{C_2}$ is a reflection of the circle (two fixed points), so $g$ reverses any orientation of $C_2$; consequently $\lk(K,C_2)$ is well defined up to sign, and we orient $C_2$ so that
\[
\ell:=\lk(K,C_2)>0 .
\]

\begin{lemma}[Linking number]\label{lem:elleven}
In the standard model of order $2^a$,
 $\ell$ is odd.
\end{lemma}

\begin{proof}
Replace $g$ by a suitable power so that it acts as rotation by $\pi$ around the axis $C_2$.
The map $\arg z_1$ fibers $\Sph\smallsetminus C_2$ over $S^1$ by disk pages, and $\ell$ equals the winding number of $\arg z_1$ along $K$. Since $g|_K$ is free of order $2$ (Lemma~\ref{lem:taxonomy}(1)) and orientation-preserving, it is conjugate to the rotation of the parameterizing circle by $\pi$. Let $A$ be the arc from a chosen base point $p$ to $g(p)$ in the positive direction of $K$; the translates $A, gA$, cover $K$. The function $\arg z_1$ increases by exactly $\pi$ under $g$, so the increase along each translate $A, gA$ equals a fixed real number $\theta$ with $\theta\equiv \pi \pmod{2\pi}$. (The path along the knot will in general wind around the knot some number of times and may backtrack. The reader is invited to trace out an arc from some $p$ to its $\pi$ rotation in Figure~\ref{fig:ehw_knot}.) Summing,
\[
2\pi\ell \;=(\text{increase along } A)+(\text{increase along } gA)\;=\;2\theta\;=\;2(\pi  + 2\pi m)
\]
for some $m\in\Z$, whence $\ell\equiv 1\pmod{2}$.
\end{proof}

\begin{remark}\label{rem:permutation}
An alternative, purely combinatorial proof that $\ell$ is odd: the number of components of a closed tangle about $C_2$ and the linking number with $C_2$ are invariant under crossing changes of $K$ with itself, so one may equivariantly change crossings until $K$ is braided about $C_2$, transverse to the pages. The ($1/2^{a-1}$)-turn symmetry exhibits the total monodromy permutation of the $\ell$ strand points as a product of $2^a$ words of the same number of letters, so it is even; connectivity of $K$ makes it an $\ell$-cycle, of sign $(-1)^{\ell-1}$. Hence $\ell$ is odd. 
\end{remark}

\section{The quotient tower}\label{sec:tower}

\begin{proposition}[One step of the tower]\label{prop:tower}
Let $a\ge2$, let $(\Sph,K,g)$ be the standard model of order $2^a$ with $K\cap C_2=\emptyset$, $g|_K$ free and orientation-preserving, and let $r:=g^{2^{a-1}}$, the $\pi$-rotation about $C_2$. Then the quotient of $\Sph$ by $\langle r\rangle$ is again $\Sph$, the quotient map $p$ is the double cover branched along $C_2$, and:
\begin{enumerate}
\item $K':=p(K)$ is a knot, and $p|_K\colon K\to K'$ is a free double cover;
\item the induced map $g'$ on the quotient is the standard rotary reflection of order $2^{a-1}$ with the same axis $C_2$ and rotation parameter $c\bmod 2^{a-1}$; $g'|_{K'}$ is free and orientation-preserving, and $K'\cap C_2=\emptyset$;
\item $\lk(K',C_2)=\lk(K,C_2)=\ell$ \textup{(}with the orientations induced by $p$\textup{)}.
\end{enumerate}
\end{proposition}

\begin{proof}
In the coordinates of the model, $r(z_1,z_2)=(-z_1,z_2)$, and the map $(z_1,z_2)\mapsto (z_1^2,z_2)$ (rescaled to the unit sphere) identifies $\Sph/\langle r\rangle$ with $\Sph$, branched along $C_2=\{z_1=0\}$. The induced symmetry sends $z_1^2\mapsto e^{2\pi i(2c)/2^a}z_1^2=e^{2\pi ic/2^{a-1}}z_1^2$ and $z_2\mapsto\bar z_2$, which is the standard model of order $2^{a-1}$ (if $a-1=1$ this is the strongly positive amphicheiral involution $(w_1,z_2)\mapsto(-w_1,\bar z_2)$ with $\Fix=S^0$), where $w_1=z_1^2$.

(1) $r|_K$ is free by Lemma~\ref{lem:taxonomy}(1), so $K'=K/\langle r\rangle$ is a circle and $p|_K$ is a free double cover. (2) $g'$ acts on the circle $K'=K/\langle r\rangle$ as the rotation induced by $g|_K$, i.e.\ with rotation number $c'\bmod 2^{a-1}$ on the quotient parameter: free, orientation-preserving, of full order $2^{a-1}$. Disjointness from $C_2$ is clear since $K\cap C_2=\emptyset$ and $p^{-1}(C_2)=C_2$.

(3) Let $D$ be a page of the open book of the quotient sphere with binding $C_2$, meeting $K'$ transversely. Its preimage $p^{-1}(D)$ is the union of the two pages of the upstairs open book lying over it, and $p$ maps $K\cap p^{-1}(D)$ bijectively onto pairs $\{x,r(x)\}$ over each point of $K'\cap D$: each intersection point downstairs has exactly two preimages, one on each upstairs page, forming a free $r$-orbit. Since $r$ preserves the orientations of $K$ and of the pages, signs are preserved, and counting an \emph{upstairs} page gives $\ell$ points with signs, matching the downstairs count: $\lk(K',C_2)=\ell$.
\end{proof}

\begin{definition}[The tower]\label{def:tower}
Given $(\Sph,K,g)$ as above with $g$ of order $2^a$, define $(\Sph,K_k,g_k)$ for $k=a,a-1,\dots,1$ by $(\Sph,K_a,g_a)=(\Sph,K,g)$ and $(\Sph,K_{k-1},g_{k-1})=$ quotient of $(\Sph,K_k,g_k)$ by $r_k:=g_k^{2^{k-1}}$, as in Proposition~\ref{prop:tower}. Set $L_k=K_k\cup C_2$ (ordered, with $C_2$ oriented so $\ell>0$), $V_k=\Sph\smallsetminus\nu(L_k)$, and $\knotpoly{k}:=\Conway_{K_k}(w)\in\Z[w]$, $\knotpoly{k}(0)=1$.
\end{definition}

By Proposition~\ref{prop:tower}, every level is again standard with the same $\ell$, and the base $K_1$ is invariant under a strongly positive amphicheiral involution.

\begin{lemma}[Base of the tower]\label{lem:HKsquare}
$\knotpoly{1}=\Conway_{K_1}=\varphi(w)^2$ for some $\varphi\in\Z[w]$ with $\varphi(0)=\pm1$.
\end{lemma}

\begin{proof}
By Hartley--Kawauchi \cite{HK}, $\Delta_{K_1}(t)\eqdot f(t)^2$ for some $f\in\Z[t]$. (If $K_1$ is the unknot the claim is trivial.) The rephrasing in terms of $\Conway$ appears in \cite{Con}; we include the one-paragraph argument. Write $\Conway_{K_1}(z)=\varepsilon\, t^{d}f(t)^2$ in $\Z[s^{\pm1}]$, $s=t^{1/2}$; evaluating at $t=1$ (i.e.\ $z=0$) gives $\varepsilon f(1)^2=1$, so $\varepsilon=1$ and $P(s):=s^{d}f(s^2)\in\Z[s^{\pm1}]$ satisfies $P^2=\Conway_{K_1}$. Applying the involution $\sigma(s)=-s^{-1}$, which fixes $\Conway_{K_1}\in\Z[w]$, gives $\sigma(P)=\pm P$. The fixed ring of $\sigma$ is $\Z[z]$ and the anti-invariants form the free $\Z[z]$-module generated by $\xi=s+s^{-1}$. In the anti-invariant case $P=\xi\,q(z)$ and $\Conway_{K_1}=(w+4)q^2$, contradicting $\Conway_{K_1}(0)=1$. Hence $P=P(z)\in\Z[z]$ with $P(z)^2\in\Z[w]$ and $P(0)^2=1$; comparing parities, $P$ is an even polynomial: $P=\varphi(w)$.
\end{proof}

\section{Level-wise identities}\label{sec:levelwise}

Throughout this section fix $1\le k\le a$ and write $L=L_k$, $V=V_k$, $\Delta=\Delta_{L_k}(t,y)$, suppressing $k$. Note $\Delta\neq0$: by \eqref{eq:torres}, $\Delta(t,1)\eqdot\qint{\ell}\knotpoly{k}\neq0$ since $\ell\neq0$.

\subsection{The covering factorization}\label{subsec:covering}

\begin{proposition}\label{prop:covering}
For $2\le k\le a$,
\[
\Delta_{L_k}(t,\,y^2)\;\eqdot\;\Delta_{L_{k-1}}(t,\,y)\cdot\Delta_{L_{k-1}}(t,\,-y),
\]
up to units $\pm t^\alpha y^\beta$.
\end{proposition}

This is a standard principle. Compare ~\cite[Corollary 4.2]{HKL}  which computes the Alexander polynomial of a cyclic cover of a knot as the product of Alexander polynomials of the knot evaluated at the $p$th roots of unity times the pth root of the variable. Here we take the double cover along the axis direction, yielding a product over the two square roots of unity similar to that story. The same mechanism underlies the classical congruences
for periodic knots and links (Fox, Murasugi) and Hartley’s treatment of free periods \cite{H2},
the orientation-preserving sibling of the present setup. Before proving Proposition~\ref{prop:covering}, we carefully set up the covering spaces and the homology of their universal abelian covers.

\subsubsection*{Fixed structures}
Write $V=V_k$, $V'=V_{k-1}$, and $p\colon V\to V'$ for the restriction of the branched covering of Proposition~\ref{prop:tower}. Since $r_k$ acts freely on $V$ (its fixed set $C_2$ has been removed) and $p^{-1}(V')=V$, the map $p$ is a free double covering of compact $3$-manifolds with boundary. Fix once and for all a triangulation $T'$ of $V'$; all covers of $V'$ below carry the lifted triangulations, and all torsions are computed from the resulting simplicial chain complexes, so that no invariance of torsion under change of CW structure is ever invoked (any two choices of $T'$ are related through common subdivisions, under which all quantities below transform identically).

Let $\phi\colon\pi_1(V')\twoheadrightarrow H_1(V')=\langle t\rangle\oplus\langle y\rangle\cong\Z^2$ be the abelianization, the generators being the oriented meridians of $K_{k-1}$ and $C_2$, and let $W\to V'$ be the corresponding universal abelian cover, with a fixed identification $\mathrm{Deck}(W/V')=\langle t,y\rangle$. Write $B:=\Z[t^{\pm1},y^{\pm1}]$ and $A:=\Z[t^{\pm1},x^{\pm1}]\subset B$ with $x:=y^2$, so that $B=A\oplus yA$ is free of rank $2$ over $A$.

\begin{lemma}[Identification of the covers]\label{lem:sharedcover}
\leavevmode
\begin{enumerate}
\item The double cover $p\colon V\to V'$ is classified by the character $\chi\colon\pi_1(V')\to\Z/2$, $\chi(t)=0$, $\chi(y)=1$ \textup{(}mod-$2$ linking with $C_2$\textup{)}; equivalently $\pi_1(V)=\phi^{-1}\langle t,y^2\rangle$.
\item On first homology, $p_*\colon H_1(V)=\langle t\rangle\oplus\langle x\rangle\to H_1(V')$ sends the meridian of $K_k$ to $t$ and the meridian of $C_2$ upstairs to $y^2$; in particular $p_*$ is injective with image $\langle t,y^2\rangle$.
\item The composite $W\to V'$ factors as $W\to V\xrightarrow{p} V'$, and $W\to V$ is the universal abelian cover of $V$, with $\mathrm{Deck}(W/V)=\langle t,y^2\rangle$ identified with $H_1(V)$ by $t\mapsto t$, $x\mapsto y^2$.
\end{enumerate}
\end{lemma}

\begin{proof}
(1) A free double quotient by the involution $r_k$ is classified by the character sending a loop to its $r_k$-parity; a meridian of $K_{k-1}$ lifts to a loop (the covering is trivial near $K$, which is disjoint from the branch locus), while a meridian of $C_2$ lifts to an arc, by the local model below. (2) Near a point of $C_2$, choose local coordinates in which $p$ is $(z_1,\zeta)\mapsto(z_1^2,\zeta)$, $z_1\in\C$, $\zeta\in\R$, with $C_2=\{z_1=0\}$. The meridian $\{|z_1|=\epsilon\}$ upstairs maps to the meridian $\{|w_1|=\epsilon^2\}$ downstairs by angle-doubling, a loop of degree $2$; orientations are preserved since $p$ is a local orientation-preserving map away from the branch locus and the meridian orientation conventions upstairs and downstairs are both induced by the (compatible) orientations of $C_2$. The meridian of $K_k$ maps homeomorphically to a meridian of $K_{k-1}$. This proves (2), and completes (1).
(3) By (1), $\pi_1(W)=\ker\phi\subset\phi^{-1}\langle t,y^2\rangle=\pi_1(V)$, so $W\to V$ is a covering with deck group $\phi(\pi_1(V))=\langle t,y^2\rangle$. The surjection $\phi|\colon\pi_1(V)\to\langle t,y^2\rangle\cong\Z^2$ factors through $H_1(V)\cong\Z^2$ via $p_*$, which is injective by (2); a surjection $\Z^2\to\Z^2$ factoring through an injection of $\Z^2$ is an isomorphism at both stages, so $\ker(\phi|_{\pi_1(V)})$ equals the commutator subgroup of $\pi_1(V)$; that is, $W\to V$ is the universal abelian cover, with the stated deck identification.
\end{proof}


\begin{proof}[Proof of Proposition~\ref{prop:covering}]
We recall the Shapiro-type theorem for twisted Alexander polynomials and torsions, \cite[Thm.~3.8]{FV}: for a finite cover $\widehat N\to N$, the invariants of $\widehat N$ agree with those of $N$ twisted by the induced representation. As called out specifically in that paper, in its simplest form, we have an epimorphism $\gamma\colon\pi_1(N)\to G$ onto a finite group, in which case the polynomial of $N$ twisted by the regular representation $\Z[G]\circ\gamma$ is the untwisted polynomial of the $\gamma$-cover. 
Apply this with $N=V_{k-1}$, $\widehat N=V_k$, $G=\Z/2$ and $\gamma=\chi$ the mod-$2$ linking character of Lemma~\ref{lem:sharedcover}(1). Two supplements are then needed. 
First, over $\Q(t,y)$ (characteristic $\ne2$) the regular representation of $\Z/2$ splits as $\mathbf1\oplus\mathrm{sgn}$, and the $\mathrm{sgn}$-twist is the substitution $y\mapsto-y$ on coefficients, giving the product $\Delta_{L_{k-1}}(t,y)\Delta_{L_{k-1}}(t,-y)$. 
Secondly, and more subtly, the cover's polynomial is produced by \cite{FV} with respect to the \emph{pulled-back} variable lattice $\phi\circ p_*$, with image $\langle t,y^2\rangle$; Lemma~\ref{lem:sharedcover}(2)--(3) identifies this sublattice with the full abelianization $H_1(V_k)=\langle t,x\rangle$ via $x=y^2$, so that no information is collapsed and the left-hand side is indeed the two-variable polynomial $\Delta_{L_k}(t,y^2)$. 
(Extension of scalars along $\Z[t^{\pm},x^{\pm}]\subset\Z[t^{\pm},y^{\pm}]$ does not alter gcd's, the only irreducible split by $y\mapsto-y$ in the Laurent ring being the unit $y$.) Note the indeterminacy in \cite[Thm.~3.8]{FV} is up to multiplication by units $\pm t^\Z y^\Z$. 
\end{proof}
\begin{corollary}\label{cor:nonvanishing}
For $1\le k\le a-1$, $\Delta_{L_k}(t,-1)\neq0$.
\end{corollary}

\begin{proof}
Specialize Proposition~\ref{prop:covering} (at level $k+1$) at $y=1$:
$\Delta_{L_{k+1}}(t,1)\eqdot\Delta_{L_k}(t,1)\,\Delta_{L_k}(t,-1)$, and the left side is $\qint{\ell}\knotpoly{k+1}\neq0$.
\end{proof}

\subsection{Symmetry and duality}

\begin{proposition}\label{prop:symmetry}
$\Delta_{L_k}(t^{-1},y)\;\eqdot\;\Delta_{L_k}(t,y)$.
\end{proposition}

\begin{proof}
The rotary reflection $g_k$ preserves $L_k$ and restricts to a self-homeomorphism of $V_k$; we compute its action on $H_1(V_k)$. For the meridian $\mu$ of $K_k$: $\lk\big(g_*\mu,K_k\big)=\lk\big(g_*\mu,g(K_k)\big)=-\lk(\mu,K_k)=-1$, since $g$ reverses the ambient orientation and preserves $K_k$ with its orientation; hence $g_*t=t^{-1}$. For the meridian $\mu_A$ of $C_2$: $g$ reverses the orientation of $C_2$, so $\lk(g_*\mu_A,C_2)=-\lk(g_*\mu_A,g(C_2))=+\lk(\mu_A,C_2)=+1$, hence $g_*y=y$. Now apply \eqref{eq:functoriality}.
\end{proof}

Combining Proposition~\ref{prop:symmetry} with Torres duality \eqref{eq:duality}:

\begin{corollary}[Axis palindromicity]\label{cor:palindromic}
$\Delta_{L_k}(t,y^{-1})\;\eqdot\;\Delta_{L_k}(t,y)$.
\end{corollary}

We emphasize that Corollary~\ref{cor:palindromic} is where the symmetry enters the algebra: duality alone inverts both variables simultaneously, and only its composition with the semilinear symmetry of Proposition~\ref{prop:symmetry} separates the axis variable.

\subsection{Palindromic normal form}

\begin{theorem}\label{thm:normalform}
Let $1\le k\le a-1$. There exists $Q_k\in\Z[u,v]$ with
\[
\Delta_{L_k}(t,y)\;\eqdot\;Q_k\big(t+t^{-1},\,y+y^{-1}\big),
\]
and $Q_k$ may be normalized \textup{(}replacing $Q_k$ by $-Q_k$ if necessary\textup{)} so that
\[
Q_k(w+2,\,2)\;=\;\qint{\ell}\;\knotpoly{k}(w) \qquad\text{exactly in }\Z[w] .
\]
\end{theorem}

\begin{proof}
Choose a representative $\Delta=\Delta_{L_k}\in\Z[t^{\pm},y^{\pm}]$. By Corollary~\ref{cor:palindromic} there are $\eta=\pm1$ and $\alpha,m\in\Z$ with $\Delta(t,y^{-1})=\eta\, t^{\alpha}y^{-m}\Delta(t,y)$. Setting $y=1$: $\Delta(t,1)=\eta t^{\alpha}\Delta(t,1)$ with $\Delta(t,1)\neq0$, so $\eta=1$, $\alpha=0$. Setting $y=-1$: $\Delta(t,-1)=(-1)^{m}\Delta(t,-1)$ with $\Delta(t,-1)\neq0$ (Corollary~\ref{cor:nonvanishing}), so $m=2e$ is even. Then $q:=y^{-e}\Delta$ satisfies $q(t,y^{-1})=q(t,y)$, hence $q\in\Z[t^{\pm}][v]$, $v=y+y^{-1}$.

Write $q=\sum_j c_j(t)v^j$. By Proposition~\ref{prop:symmetry}, $q(t^{-1},y)=\eta' t^{\alpha'} y^{\beta'} q(t,y)$ for some unit; comparing $v$-expansions forces $\beta'$ to vanish and gives $c_j(t^{-1})=\eta' t^{\alpha'}c_j(t)$ for all $j$, with a single fixed unit $\eta't^{\alpha'}$. At $y=1$: $q(t,1)=\pm\Delta(t,1)\eqdot\qint\ell\knotpoly{k}$, which is nonzero and, evaluated at $t=1$, equals $\pm\ell\ne0$; from $q(t^{-1},1)=\eta' t^{\alpha'}q(t,1)$ at $t=1$ we get $\eta'=1$. Thus $q(t^{-1},1)=t^{\alpha'}q(t,1)$. Assuming $\alpha'$ is even, 
replacing $q$ by $t^{-\alpha'/2}q$ gives us that each $c_j(t^{-1})=c_j(t)$, i.e.\ $c_j\in\Z[u].$ 
($\alpha'$ is even because $q(t,1)= \pm t^\beta\qint{\ell}p_k$ with $\qint{\ell}p_k$ symmetric. Substituting into $q(t^{-1},1)=t^{\alpha'}q(t,1)$ yields $\alpha'=-2\beta$.) This yields $Q_k\in\Z[u,v]$ with $\Delta\eqdot Q_k(u,v)$.

Finally, $Q_k(w+2,2)=q(t,1)=\pm\qint{\ell}\knotpoly{k}$ as elements of $\Z[w]$: two elements of $\Z[w]$ that agree up to $\pm t^\alpha$ agree up to sign, since applying $t\mapsto t^{-1}$ to $p=\pm t^{\alpha}p'$ with $p,p'\in\Z[w]$ gives $t^{2\alpha}=1$. Replace $Q_k$ by $-Q_k$ if needed.
\end{proof}

\begin{definition}\label{def:R}
With $Q_k$ normalized as in Theorem~\ref{thm:normalform}, set
\[
R_k(w):=Q_k(w+2,\,-2)\in\Z[w],\qquad \widehat R_k:=\eps_\ell\,R_k .
\]
Note $\Delta_{L_k}(t,-1)\eqdot\pm R_k$.
\end{definition}

\begin{lemma}\label{lem:Rcongruence}
$R_k\;\equiv\;\qint{\ell}\,\knotpoly{k} \pmod 4$.
\end{lemma}

\begin{proof}
Write $Q_k(u,v)=\sum_j q_j(u)v^j$. Then
\[
Q_k(u,2)-Q_k(u,-2)=\sum_{j\ \mathrm{odd}}q_j(u)\big(2^j-(-2)^j\big)=\sum_{j\ \mathrm{odd}}2^{\,j+1}q_j(u)\equiv0\pmod4,
\]
since $j\ge1$. Now use the normalization $Q_k(w+2,2)=\qint\ell\knotpoly{k}$.
\end{proof}

\section{Telescoping}\label{sec:telescope}

\begin{theorem}[Integral factorization]\label{thm:telescope}
For $1\le k\le a-1$,
\[
\knotpoly{k+1}\;=\;\eps_\ell\,\knotpoly{k}\,R_k\;=\;\knotpoly{k}\,\widehat R_k \qquad\text{in }\Z[w],
\]
with $R_k(0)=\eps_\ell$, $\widehat R_k(0)=1$. Consequently
\[
\Conway_K=\knotpoly{a}=\varphi(w)^2\,\prod_{k=1}^{a-1}\widehat R_k(w) .
\]
\end{theorem}

\begin{proof}
Specializing Proposition~\ref{prop:covering} at $y=1$ and applying \eqref{eq:torres} on the left,
\[
\qint{\ell}\,\knotpoly{k+1}\;\eqdot\;\Delta_{L_k}(t,1)\cdot\Delta_{L_k}(t,-1)\;\eqdot\;\big(\qint{\ell}\knotpoly{k}\big)\cdot R_k .
\]
Both extremes lie in $\Z[w]$; two elements of $\Z[w]$ agreeing up to $\pm t^{\alpha}y^{\beta}$ agree up to sign (specialize $y$, then argue as at the end of the proof of Theorem~\ref{thm:normalform}). Cancelling $\qint\ell\neq0$ in the domain $\Z[w]$:
\[
\knotpoly{k+1}=s_k\,\knotpoly{k}\,R_k,\qquad s_k\in\{\pm1\}.
\]
Evaluating at $w=0$: $1=s_k\,R_k(0)$, so $R_k(0)=s_k\in\{\pm1\}$. By Lemma~\ref{lem:Rcongruence}, $R_k(0)\equiv\ell\pmod4$; since $\pm1$ are distinct mod $4$ and $\ell$ is odd, $R_k(0)=\eps_\ell$ and $s_k=\eps_\ell$. Telescoping from $\knotpoly{1}=\varphi^2$ (Lemma~\ref{lem:HKsquare}) gives the product formula.
\end{proof}

\begin{theorem}[Mod $4$]\label{thm:mod4}
$\knotpoly{k+1}\equiv\eps_\ell\qint{\ell}\,\knotpoly{k}^{\,2}\pmod4$ for $1\le k\le a-1$, hence
\[
\Conway_K\;\equiv\;\big(\eps_\ell\qint{\ell}\big)^{2^{a-1}-1}\,\varphi^{\,2^{a}}\pmod4,
\]
and $\Conway_K\bmod4\in S$; that is, $\Conway_K\equiv f(z)f(-z)\pmod4$ for some $f\in\Z[z]$.
\end{theorem}

\begin{proof}
By Theorem~\ref{thm:telescope} and Lemma~\ref{lem:Rcongruence}, $\knotpoly{k+1}=\eps_\ell\knotpoly{k}R_k\equiv\eps_\ell\qint\ell\knotpoly{k}^2\pmod4$. Induction gives $\knotpoly{k}\equiv(\eps_\ell\qint\ell)^{2^{k-1}-1}(\varphi^2)^{2^{k-1}}$; take $k=a$. For membership in $S$: if $a=1$ this is $\varphi^2\in S$ (Lemma~\ref{lem:normform}). If $a\ge2$, the exponent $2^{a-1}-1$ is odd, so
\[
\big(\eps_\ell\qint\ell\big)^{2^{a-1}-1}=\big(\eps_\ell\qint\ell\big)\cdot\Big(\big(\eps_\ell\qint\ell\big)^{(2^{a-1}-2)/2}\Big)^{2},
\]
a product of $\eps_\ell\qint\ell\in S$ (Lemma~\ref{lem:qint}) and squares; conclude by multiplicativity of $S$ (Lemma~\ref{lem:normform}).
\end{proof}

\begin{corollary}\label{cor:order4}
 Suppose $a\geq 2$. With $M_{a-1}:=\tfrac12\big(Q_{a-1}(u,2)+Q_{a-1}(u,-2)\big)$ and $E_{a-1}:=\tfrac14\big(Q_{a-1}(u,2)-Q_{a-1}(u,-2)\big)$, both in $\Z[u]=\Z[w]$ and abbreviated $M,E$ below,
\[
\eps_\ell\,\qint{\ell}\,\Conway_K\;=\;M_{a-1}^2-4E_{a-1}^2 ,
\]
with $M\equiv\qint\ell\,\knotpoly{a-1}\pmod2$; in particular $M(0)$ is odd, and comparing constant terms modulo $8$,
\[
\eps_\ell\,\ell\;\equiv\;1-4E(0)^2\pmod 8,\qquad\text{i.e.}\qquad E(0)\equiv\tfrac{\ell^2-1}{8}\pmod2 .
\]
\end{corollary}

\begin{proof}
Take $k=a$: from the proof of Theorem~\ref{thm:telescope}, \[\eps_\ell\qint\ell\knotpoly{k}=\big(\qint\ell\knotpoly{k-1}\big)R_{k-1}=Q_{k-1}(u,2)\,Q_{k-1}(u,-2)=M^2-4E^2.\] Writing $Q_{k-1}=\sum q_jv^j$, $M=\sum_{j\,\mathrm{even}}2^jq_j$ and $E=\sum_{j\,\mathrm{odd}}2^{\,j-1}q_j$ lie in $\Z[u]$, and $M\equiv Q_{k-1}(u,2)=\qint\ell\knotpoly{k-1}\pmod2$, whose constant term is $\ell$, odd. Evaluating the displayed identity at $w=0$ gives $\eps_\ell\ell=M(0)^2-4E(0)^2$ with $M(0)$ odd, so $M(0)^2\equiv1\pmod8$ and $\eps_\ell\ell\equiv1-4E(0)^2\pmod8$; the two cases $E(0)$ even/odd give $\eps_\ell\ell\equiv1$ resp.\ $5\pmod 8$, which matches $\ell\equiv\pm1$ resp.\ $\pm3\pmod8$, i.e.\ $E(0)\equiv(\ell^2-1)/8\pmod2$ (the second supplement to quadratic reciprocity in disguise).
\end{proof}
\begin{remark}
When $a=1$ we are in the Hartley-Kawauchi square case, where $\nabla_K$ is a square, with no quantum integer correction term.
\end{remark}

\begin{lemma}\label{lem:conway_sub}
Let $a,m\geq1$ be integers and let $f(w)\in \mathbb Z[w]$. Then 
\begin{enumerate}
\item $f(w)|_{t\mapsto t^m} = f([m]^2w)$; note $[m]^2w\in\Z[w]$ for every $m$, odd or even. 
\item If $a$ and $m$ are odd, $[a]_{t\mapsto t^m} =[am]/[m]$.
\end{enumerate}
\end{lemma}
\begin{proof}
Note that $z=t^{1/2}-t^{-1/2}\mapsto t^{m/2}-t^{-m/2}$ under the substitution $t\mapsto t^m$. Applying this to $z$ we get $t^{m/2}-t^{-m/2}=(t^{1/2}-t^{-1/2})\frac{t^{m/2}-t^{-m/2}}{t^{1/2}-t^{-1/2}}=[m]z$. Hence $w\mapsto [m]^2w.$

Similarly $[a]_{t\mapsto t^m}=[am]/[m]$ by inspection of the formulas.
\end{proof}

\begin{theorem}\label{thm:conjecture}
Conjecture~\ref{conj:main} holds for all amphicheiral knots: $\nabla_K\equiv f(z)f(-z)$ in
$\Z_4[z]$. In fact we have the stronger integer statements:
\begin{enumerate}
\item If $K$ is negative amphicheiral, its Conway polynomial splits over the
integers: $\nabla_K=f(z)f(-z)$, $f\in\Z[z]$ \textup{(Hartley)}.
\item If $K$ is strongly positive amphicheiral, $\nabla_K=f(w)^2$ in $\Z[w]$
\textup{(Hartley--Kawauchi)}.
\item If $K$ is periodically positive amphicheiral of order $2^a$, $a\geq2$, then
$\eps_\ell[\ell]\,\nabla_K=M^2-4E^2$ in $\Z[w]$, where $\ell$ is the linking number
of $K$ with the axis.
\item For any positive amphicheiral knot there is an exact identity in $\Z[w]$,
\[
\Big(\prod_{i}\eps_{\ell_i}\,[m_i\ell_i]\Big)\,\nabla_K
\;=\;\Big(\prod_{i}[m_i]\Big)\big(M^2-4E^2\big),
\]
where the products run over a finite \textup{(}possibly empty\textup{)} index set,
$\ell_i$ and $m_i$ are odd positive integers \textup{(}axis linking numbers and
winding numbers of the invariant companions\textup{)}, and $M(0)$ is odd.
\end{enumerate}
\end{theorem}

\begin{proof}
Statements (1) and (2) are due to Hartley \cite{H1} and Hartley--Kawauchi \cite{HK}, and (3)
is Corollary~\ref{cor:order4}. We prove (4) in four steps, then deduce the mod 4 statement.

\emph{Step 1: the equivariant companionship factorization.}
Let $h$ be an orientation-reversing self-homeomorphism of $(S^3,K)$ preserving the
orientation of $K$. Following Hartley's analysis \cite[\S3]{H1}, $h$ induces a symmetry
of the companionship structure of $K$: in the satellite formula for the Alexander
polynomial, which expresses $\Delta_K$ up to units as a product of polynomials
$\Delta_J(t^{\,n})$ over the companions $J$ of $K$ with winding numbers $n$, the
symmetry permutes the companions, preserving winding numbers. Companions
interchanged with a mirror partner contribute equal factors, $\Delta$ being
insensitive to mirroring, so their total contribution is a perfect square;
companions of winding number $0$ contribute the unit $\Delta_J(1)=\pm1$ and are
dropped. A companion $J$ fixed by the symmetry, with invariant companion solid
torus $V$ and winding $m\neq0$, is itself positive amphicheiral: $h_*[K]=[K]$ and
$[K]=m[J]$ in $H_1(V)\cong\Z$ force $h_*[J]=[J]$, so the induced symmetry of
$(S^3,J)$ reverses the ambient orientation while preserving that of $J$. These fixed atoms may be taken to be atoroidal by the reduction argument of
\cite[Thm.~2.6]{CM}, hence hyperbolic (torus knots are chiral), so that by Mostow rigidity their amphicheiral
symmetry may be chosen of finite order; moreover their winding numbers are odd \cite{H1}. The upshot is a factorization, up to units $\pm t^{\Z}$,
\[
\Delta_K(t)\;\eqdot\;\prod_{j\in J}\Delta_{J_j}\!\big(t^{\,n_j}\big)^{2}\cdot
\prod_{i\in I}\Delta_{K_i}\!\big(t^{\,m_i}\big),
\]
where each $K_i$ is a nontrivial positive amphicheiral knot admitting an
orientation-reversing diffeomorphism of finite order preserving its orientation,
and each $m_i$ is odd.

\emph{Step 2: normalization.}
We convert this to an exact identity of Conway polynomials. For any knot $J$ and
any $n\geq1$, the substitution $t\mapsto t^{\,n}$ carries $w=t-2+t^{-1}$ to
$t^{\,n}-2+t^{-n}= \lambda_n(w)\in\Z[w]$, where $\lambda_n(w)=[n]^2w$ (Lemma~\ref{lem:conway_sub}), so
$\Delta_J(t^{\,n})\eqdot\nabla_J(\lambda_n(w))$, an element of $\Z[w]$ taking the
value $\nabla_J(0)=1$ at $w=0$; note this holds for every $n$, odd or even (only
the quantum-integer transformation law of Lemma~\ref{lem:conway_sub} requires oddness). Hence both
sides of the displayed factorization are, up to units, elements of $\Z[w]$; two
elements of $\Z[w]$ agreeing up to $\pm t^{\alpha}$ agree up to sign (one can argue as at
the end of the proof of Theorem~\ref{thm:normalform}), and evaluating at $w=0$, where every factor
equals $1=\nabla_K(0)$, fixes the sign:
\[
\nabla_K \;=\; P(w)^2\cdot\prod_{i\in I}\nabla_{K_i}\!\big(\lambda_{m_i}(w)\big),
\qquad P:=\prod_{j}\nabla_{J_j}\circ\lambda_{n_j}\in\Z[w],\quad P(0)=1.
\tag{$\star$}
\]

\emph{Step 3: the atoms.}
Fix $i\in I$ and apply Lemma~\ref{lem:taxonomy} to $K_i$ with its finite-order symmetry: after
passing to an odd power, $K_i$ is either strongly positive amphicheiral or
invariant under a standard rotary reflection of order $2^{a_i}$ with $a_i\geq2$.
In the first case $\nabla_{K_i}=\varphi_i(w)^2$ by Lemma~\ref{lem:HKsquare}, and we transfer $i$
out of $I$ into the square factor, replacing $P$ by
$P\cdot(\varphi_i\circ\lambda_{m_i})$. Let $I^{+}\subseteq I$ be the indices that
remain; for $i\in I^{+}$, Corollary~\ref{cor:order4} supplies $M_i,E_i\in\Z[w]$ with $M_i(0)$
odd and
\[
\eps_{\ell_i}\,[\ell_i]\,\nabla_{K_i}\;=\;M_i^2-4E_i^2,
\]
where $\ell_i>0$ is the linking number of $K_i$ with its axis, odd by Lemma~\ref{lem:elleven}.

\emph{Step 4: substitution and assembly.}
Apply $t\mapsto t^{\,m_i}$ to the identity of Step 3. By Lemma~\ref{lem:conway_sub},
$[\ell_i]\mapsto[m_i\ell_i]/[m_i]$ and $w\mapsto\lambda_{m_i}(w)$, so, clearing the
denominator,
\[
\eps_{\ell_i}\,[m_i\ell_i]\;\nabla_{K_i}\!\big(\lambda_{m_i}(w)\big)
\;=\;[m_i]\,\big(\widetilde M_i^{\,2}-4\widetilde E_i^{\,2}\big),
\qquad
\widetilde M_i:=M_i\circ\lambda_{m_i},\ \ \widetilde E_i:=E_i\circ\lambda_{m_i},
\]
an exact identity in $\Z[w]$; here $[m_i]$ and $[m_i\ell_i]$ lie in $\Z[w]$ because
$m_i$ and $m_i\ell_i$ are odd. Multiplying $(\star)$ through by
$\prod_{i\in I^{+}}\eps_{\ell_i}[m_i\ell_i]$ and substituting,
\[
\Big(\prod_{i\in I^{+}}\eps_{\ell_i}[m_i\ell_i]\Big)\nabla_K
\;=\;\Big(\prod_{i\in I^{+}}[m_i]\Big)\cdot
P^2\prod_{i\in I^{+}}\big(\widetilde M_i^{\,2}-4\widetilde E_i^{\,2}\big).
\]
The set $\{M^2-4E^2: M,E\in\Z[w]\}$ is closed under multiplication by
Brahmagupta's identity
$(M^2-4E^2)((M')^2-4(E')^2)=(MM'+4EE')^2-4(ME'+EM')^2$, and absorbs the square:
$P^2(M^2-4E^2)=(PM)^2-4(PE)^2$. This yields (4), with $M(0)$ odd since each
$\widetilde M_i(0)=M_i(0)$ is odd, $P(0)=1$, and the Brahmagupta composition
preserves oddness of the constant term modulo the visible factor of $4$. If
$I^{+}=\emptyset$, the identity holds with empty products, $M=P$ and $E=0$.

\emph{The mod 4 statement.}
For negative amphicheiral knots the splitting is integral by (1), and squares lie
in $S$ by Lemma~\ref{lem:normform}, which covers (2); it remains to treat (4). Multiply both
sides of (4) by $\prod_i\eps_{m_i}$ and reduce modulo 4. Since
$\eps_{m}\eps_{\ell}=\eps_{m\ell}$ for odd $m,\ell$, the left side becomes
$s\,\nabla_K$ with $s:=\prod_i\eps_{m_i\ell_i}[m_i\ell_i]$, which lies in $S$ by
Lemma~\ref{lem:qint} and multiplicativity (Lemma~\ref{lem:normform}); the right side is the product of
$\prod_i\eps_{m_i}[m_i]\in S$ and $M^2-4E^2\equiv M^2\in S$, hence lies in $S$.
Thus $s\,\nabla_K\in S$, and $s$ has the odd integer $\prod_i m_i\ell_i$ as its
constant term, so Lemma~\ref{lem:reduce} gives $\nabla_K\in S$.
\end{proof}

We also prove this beautiful cyclotomic decomposition for $\Conway_K.$

\begin{theorem}\label{thm:cyclores}
Let $K$ be a $2^a$ periodically positive amphicheiral knot in the standard rotary model,
$a\geq2$, with axis linking number $\ell$, and let $Q_1(u,v)$ be the normal form
\textup{(}Theorem~\textup{\ref{thm:normalform}}\textup{)} of the first quotient link
$L_1$. Then
\begin{equation}\label{eq:cyclores}
\eps_\ell\qint\ell\,\Conway_K
\;=\;\prod_{\zeta^{2^{a-1}}=1} Q_1\big(u,\ \zeta+\zeta^{-1}\big)
\;=\;Q_1(u,2)\,Q_1(u,-2)\prod_{k=1}^{a-2}\Big(\mathrm{N}_k\,Q_1(u,\lambda_k)\Big)^{2},
\end{equation}
where $\lambda_k=2\cos(\pi/2^{k})=\zeta_{2^{k+1}}+\zeta_{2^{k+1}}^{-1}$ for $\zeta_b$ a
primitive $b$-th root of unity, and $\mathrm{N}_k$ denotes the coefficient-wise Galois
norm from $\Z[\lambda_k]$ to $\Z$, that is,
\[\mathrm{N}_kQ_1(u,\lambda_k)=\prod_{\substack{0<j<2^{k}\\ j\ \mathrm{odd}}}
Q_1\big(u,\,2\cos(j\pi/2^{k})\big)\in\Z[u].\]
\end{theorem}

Note that for $a=2$ the product is $Q_1(u,2)Q_1(u,-2)=(M+2E)(M-2E)$, recovering
Corollary~\ref{cor:order4} in the $a=2$ case. The first two norms are: $\mathrm{N}_1=\operatorname{id}$ , while
$\mathrm{N}_2Q_1(u,\sqrt2)=Q_1(u,\sqrt2)\,Q_1(u,-\sqrt2)$.

The heart of the proof is the following covering formula, in which the regular
representation of the deck group appears explicitly: there is one factor on the right
for each character of $\Z_{2^{a-1}}$.

\begin{lemma}\label{lem:regrep}
With $N=2^{a-1}$,
\[
\Delta_{L_a}(t,y^{N})\;\eqdot\;\prod_{\zeta^{N}=1}\Delta_{L_1}(t,\zeta y),
\]
where $\eqdot$ denotes equality up to $\pm t^iy^j$.
\end{lemma}

\begin{proof}
A clean proof with minimal extra machinery involves $(a-1)$ applications of
Proposition~\ref{prop:covering}, by induction on $a$: writing
$-y^{2^{a-2}}=(\omega y)^{2^{a-2}}$ for $\omega$ a primitive $2^{a-1}$-st root of unity
merges the two half-products into the full one.
\end{proof}

\begin{remark}
Alternatively, as in the proof of Proposition~\ref{prop:covering}, the lemma is an application of the Shapiro-type theorem of \cite[Thm.~3.8]{FV}, with the group $\Z_2$ replaced by $\Z_{2^{a-1}}$ and its regular representation, though the details are a little fiddly. This is similar to~\cite[Corollary 4.2]{HKL}.
\end{remark}

\begin{proof}[Proof of Theorem~\ref{thm:cyclores}]
Set $y=1$ in Lemma~\ref{lem:regrep}:
$\Delta_{L_a}(t,1)\eqdot\prod_{\zeta^{N}=1}\Delta_{L_1}(t,\zeta)$. On the left, Torres
\eqref{eq:torres} gives $\Delta_{L_a}(t,1)\eqdot\frac{t^\ell-1}{t-1}\Delta_K(t)$,
which after symmetrization is $\qint\ell\,\Conway_K$. On the right, by
Theorem~\ref{thm:normalform} there are $c,d\in\Z$ and a sign $\sigma$ with
$\sigma t^cy^d\,\Delta_{L_1}(t,y)=Q_1(u,y+y^{-1})$; evaluating at $y=\zeta$ and taking
the product over $\zeta$,
\[
\prod_{\zeta^{N}=1}\Delta_{L_1}(t,\zeta)
=\sigma^{-N}t^{-Nc}\Big(\prod_\zeta\zeta\Big)^{-d}\,\Pi(u),
\qquad
\Pi(u):=\prod_{\zeta^{N}=1}Q_1\big(u,\zeta+\zeta^{-1}\big),
\]
and $\Pi\in\Z[u]$ since the product is Galois-stable. Hence
$\qint\ell\,\Conway_K=\delta\,t^{i}\,\Pi(u)$ for some sign $\delta$ and some $i$; both
sides are symmetric under $t\mapsto t^{-1}$, which forces $i=0$.

It remains to see $\delta=\eps_\ell$, which we do by evaluating at $t=1$. There the
left side equals $\ell$. On the
right, Torres applied to $L_1$ in the axis variable (using
$\operatorname{lk}(K_1,C_2)=\ell$ and $\Delta_{C_2}=1$), together with the
normalization $Q_1(\,\cdot\,,2)\big|_{t=1}=\ell\,\Conway_{K_1}(0)=\ell>0$, gives
\[
Q_1(u,v)\big|_{t=1}\;=\;c_\ell(v),\qquad
c_\ell(y+y^{-1}):=\sum_{j=-(\ell-1)/2}^{(\ell-1)/2}y^{\,j}
=y^{-(\ell-1)/2}\,\frac{y^{\ell}-1}{y-1}.
\]
The factor at $\zeta=1$ contributes $c_\ell(2)=\ell$. Since $\ell$ is odd (Lemma~\ref{lem:elleven}), $\zeta\mapsto\zeta^{\ell}$ permutes the
nontrivial $N$-th roots of unity, so
$\prod_{\zeta\neq1}\frac{\zeta^{\ell}-1}{\zeta-1}=1$, while
$\prod_{\zeta\neq1}\zeta^{-(\ell-1)/2}=(-1)^{(\ell-1)/2}=\eps_\ell$ because
$\prod_{\zeta\neq1}\zeta=-1$ for $N$ even. Hence $\Pi\big|_{t=1}=\eps_\ell\,\ell$,
forcing $\delta=\eps_\ell$ and proving the first equality of \eqref{eq:cyclores}.

For the second, partition the $N$-th roots of unity by exact order: $\zeta=1$ and
$\zeta=-1$ contribute $Q_1(u,2)$ and $Q_1(u,-2)$; for $1\leq k\leq a-2$ the
$\varphi(2^{k+1})=2^{k}$ primitive $2^{k+1}$-st roots satisfy $\zeta\neq\zeta^{-1}$ and
map two-to-one under $\zeta\mapsto\zeta+\zeta^{-1}$ onto the $2^{k-1}$ Galois
conjugates $2\cos(j\pi/2^{k})$, $j$ odd, of $\lambda_k$, contributing
$\big(\mathrm{N}_kQ_1(u,\lambda_k)\big)^{2}$.
\end{proof}

\begin{remark}\label{rem:characters}
The right side of \eqref{eq:cyclores} has one factor per character of the deck group:
the trivial character contributes $Q_1(u,2)=\qint\ell\,\Conway_{K_1}$, the quotient
knot; the sign character contributes $Q_1(u,-2)=R_1$; and the complex characters pair
under conjugation, each pair contributing a real norm. Comparing with the telescope of
Theorem~\ref{thm:telescope} identifies
$\widehat R_{k+1}=\big(\mathrm{N}_kQ_1(u,\lambda_k)\big)^{2}$ for $k\geq1$: the fact
that every correction factor above the first is a perfect square is, from this
viewpoint, complex conjugation acting freely on the nonreal characters.
\end{remark}

\begin{remark}\label{rem:braided}
For the braided families of Appendices~\ref{app:burau} and~\ref{app:fivestrand} the
theorem is an eigenvalue identity: with $N_\alpha=\bar\psi_\alpha$ the level-$1$ Burau
monodromy, $\eps_\ell\qint\ell\,\Conway_{K_a}=\det\big(N_\alpha^{2^{a-1}}-I\big)$, and
factoring $\lambda^{2^{a-1}}-1=\prod_{\zeta^{2^{a-1}}=1}(\lambda-\zeta)$ over the
eigenvalues and recentering each $\det(N_\alpha-\zeta I)$ reproduces
\eqref{eq:cyclores} factor by factor.
\end{remark}

\section{Examples}\label{sec:examples}

\begin{figure}
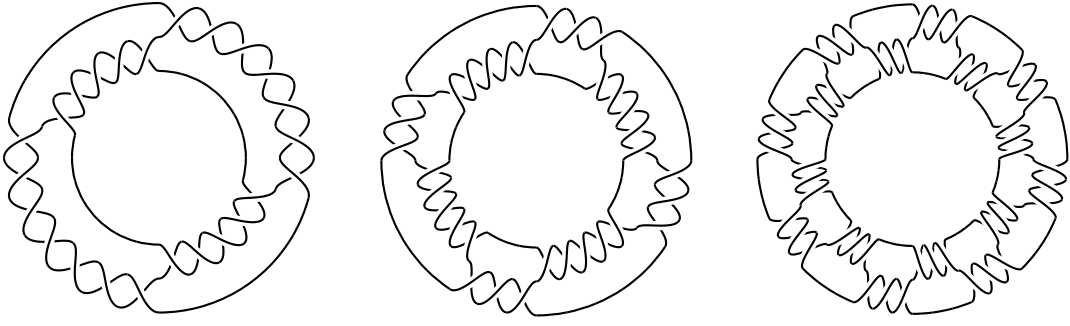

\braidwheel{2}{-1,-1,-1,-1,-1,2,2,2,2,2,2,2}\qquad
\braidwheel{4}{-1,-1,-1,-1,-1,2,2,2}\qquad \braidwheel{8}{1,1,1,2,2,2}
\caption{Positive amphicheiral braids of orders 2, 4 and 8, with generating words $\sigma_1^{-5}\sigma_2^7$, $\sigma_1^{-5}\sigma_2^3$ and $\sigma_1^{3}\sigma_2^3$ respectively.}\label{fig:braidexamples}
\end{figure}

\subsection{The braid families of Conant--Manathunga}\label{subsec:fibonacci}
The order-$4$ examples of \cite{CM} are closures of $3$-braids $\sigma_1^n\sigma_2^m\sigma_1^{-n}\sigma_2^{-m}\sigma_1^n\sigma_2^m\sigma_1^{-n}\sigma_2^{-m}$ about the axis $C_2$. See the left of Figure~\ref{fig:braidexamples} for an example with $n=-5, m=3$. An order $8$ example is also shown.
In \cite{CM} (recomputed in this paper's Appendix), it is shown that
\[
\Conway_K \;=\; X\,\big(4-X(w+3)\big),\qquad X=F_n^2F_m^2,
\]
where $F_j$ denotes the Fibonacci polynomials ($F_1=1$, $F_2=z$, $F_{j+1}=zF_j+F_{j-1}$) and $n,m$ are odd (so that $F_nF_m\in\Z[w]$ and $\Conway_K(0)=1$). Here $\ell=3$, $\eps_\ell=-1$, $\qint3=w+3$. Writing $T:=(w+3)X$:
\[
\Conway_K=X(4-T)=\knotpoly{1}\cdot\widehat R_1,\qquad \knotpoly{1}=X=(F_nF_m)^2,\quad \widehat R_1=4-T,
\]
which is precisely the factorization of Theorem~\ref{thm:telescope} with base square $\varphi=F_nF_m$ and one correction factor. Consistency checks: $\widehat R_1(0)=4-3=1$; $\widehat R_1=4-T\equiv-\qint3\,\knotpoly{1}=\eps_\ell\qint\ell\knotpoly{1}\pmod4$ (Lemma~\ref{lem:Rcongruence}); and Corollary~\ref{cor:order4} reads
\[
-\qint3\,\Conway_K=T(T-4)=(T-2)^2-4,\qquad M=T-2,\ E=1 .
\]
The corresponding two-variable polynomial is forced: $Q_1(u,v)=v+(T-2)$, i.e.
\[
\Delta_{L_1}(t,y)\;\eqdot\;y+y^{-1}+(w+3)F_n^2F_m^2-2,
\]
with $Q_1(u,2)=T=\qint3\knotpoly{1}$ (Torres) and $Q_1(u,-2)=T-4=R_1$, $R_1(0)=-1=\eps_\ell$. Proposition~\ref{prop:covering} then \emph{predicts} the two-variable polynomial of the original link:
\[
\Delta_{K\cup C_2}(t,x)\;\eqdot\;x^2-\big((T-2)^2-2\big)\,x+1 .
\]
Both displayed two-variable formulas are directly testable against the Burau-matrix computations of \cite{CM}; the second is equivalent to the statement that the reduced Burau matrix (with axis variable) of the symmetrized braids has trace $(T-2)^2-2$ and determinant $1$, and the first to the corresponding statement one level down, trace $T-2$. Conversely, the theorem \emph{explains} the shape of the formula of \cite{CM}: the Fibonacci squares are the Hartley--Kawauchi square of the quotient knot $K_1$, the closure of $\sigma_1^n\sigma_2^m\sigma_1^{-n}\sigma_2^{-m}$, whose Conway polynomial is predicted to be $F_n^2F_m^2$. This is indeed the case and follows directly from the Burau calculations in the proof of \cite{CM}, Proposition 3.13.

In an appendix we step through the computations for the various Alexander polynomials and verify the predictions.

\subsection{The Ermotti--Hongler--Weber knot}\label{subsec:EHW}
The counterexample of \cite{EHW} to the integral splitting \eqref{eq:split} (see Figure~\ref{fig:ehw_knot}) carries an orientation-reversing symmetry of order $4$, with axis linking number $\ell=3$. As pictured, this is not obviously invariant under the standard rotary symmetry where we rotate and then reflect through the xy-plane.
It is more obviously invariant under a symmetry which rotates the first two coordinates and then reflects through the origin. It can be converted to standard form by rotating two of four repeating blocks by 180 degrees around the circular axis $C_1$. More precisely, represent the given knot as the closure of a tangle $B B' BB'$, where $B$ is some choice of tangle capturing $90$ degrees of the diagram and $B'$ is the tangle reflection reversing strand order. Let $\beta$ be the fundamental braid on three strands where $\beta^2$ is a full twist generating the center of the braid group. The knot is isotopic to the closure of $B (\beta^{-1} \bar B \beta) B(\beta^{-1}\bar B\beta)$, where $\bar B$ is the reflection through the $xy$ plane. Conjugate to move $\beta$ to the front and note $\beta^{-1}=\bar \beta$, so the knot is the closure of $(\beta B)\overline{\beta B}(\beta B)\overline{\beta B}$, which is indeed invariant under the standard rotary reflection.

\begin{figure}
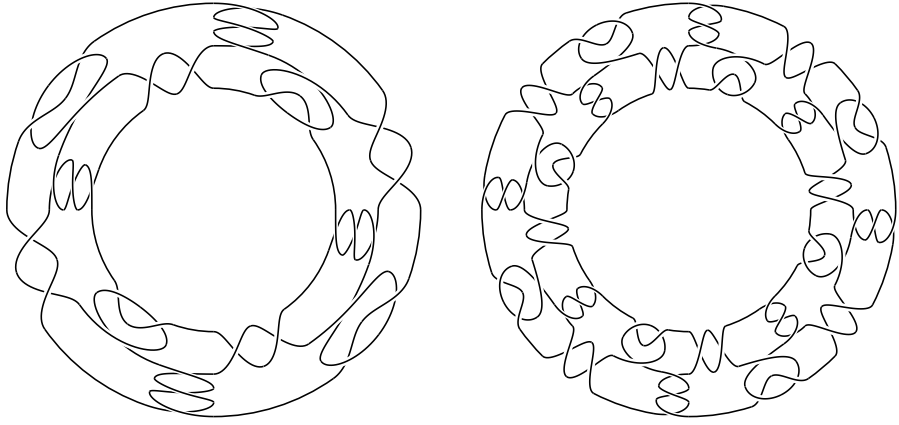

    \centering
    \scalebox{0.75}{\hvwheel{4}{\JimCell}}\qquad
    \scalebox{0.75}{\hvwheel{8}{\JimCell}}
    \caption{The EHW knot (left)is positive amphicheiral with generating symmetry of order $4$. An order 8 variant is pictured on the right. }
    \label{fig:ehw_knot}
\end{figure}
 According to \cite{EHW}, the Conway polynomial is
\[
\Conway_K=(1-w)\,(2w^2-1)^2\,P(w),\qquad P=4w^4+16w^3+12w^2-16w+1
\]
Corollary~\ref{cor:order4} then asserts an exact identity $-\qint3\Conway_K=M^2-4E^2$, and indeed, using
\[
(w+3)(w-1)=(w+1)^2-4,\qquad P=\big(2w^4{+}8w^3{+}6w^2{-}8w{+}1\big)^2-4\big(w^4{+}4w^3{+}3w^2{-}4w\big)^2,
\]
and composing norms as in Lemma~\ref{lem:normform} (for the form $M^2-4E^2$),
\[
-\qint3\,\Conway_K \;=\; M^2-4E^2,\qquad
\begin{aligned}
M&=(2w^2-1)\big(2w^5+14w^4+30w^3+10w^2-23w+1\big),\\
E&=(2w^2-1)\big(w^5+7w^4+15w^3+5w^2-12w+1\big).
\end{aligned}
\]
 Modulo $4$ it collapses to $-\qint3\Conway_K\equiv M^2$, consistent with Theorem~\ref{thm:mod4}. 

\section{Questions}\label{sec:questions}

\begin{question}
Theorem~\ref{thm:A} produces, for each level, a $y$-palindromic two-variable polynomial $Q_k$, well defined up to the stated normalization. Is there an intrinsic normalization of a ``half'' invariant refining it, in the spirit of the half-Conway polynomial of Boyle--Chen \cite{BC} for strongly negative amphicheiral knots, together with an equivariant skein relation computing it for rotary reflections of order $2^a$?
\end{question}

\begin{question}
Which congruences modulo $8$ (and higher powers of $2$) on $\Conway_K$ follow from the existence of the presentation $\eps_\ell\qint\ell\Conway_K=M^2-4E^2$ of Corollary~\ref{cor:order4}, beyond the constant-term constraint $E(0)\equiv(\ell^2-1)/8\pmod2$? Is $E$ (for a suitable normalization) a well-defined invariant of the symmetric knot, with a topological interpretation? (For the Fibonacci families $E=1$; for the EHW knot $E$ has degree $7$ in $w$.) More generally, which congruences follow from the full integral factorization of Theorem~\ref{thm:telescope}?
\end{question}

\begin{question}
The counterexamples of \cite{CM} to Stoimenow's conjectures concern the leading coefficient of $\Conway_K$ for amphicheiral knots. Does Theorem~\ref{thm:telescope} impose any constraints on the leading coefficients of $\Conway_K$ weaker than the ones conjectured by Stoimenow?
\end{question}

\begin{question}
The mod $4$ Kawauchi conjecture is equivalent to the statement that a series of degree $4n$ Vassiliev invariants, $pc_{4n}$, constructed from Conway polynomial coefficients, vanishes on amphicheiral knots \cite{Con}. This follows from the stronger conjecture that each $pc_{4n}$ is congruent modulo $2$ to an integer valued degree $4n-1$ invariant, $v_{4n-1}$, that reverses sign under mirror image. Now that the consequence of the stronger conjecture has been verified, can we shed any light on what the mysterious $v_{4n-1}$ are?
\end{question}

\appendix
 
\section{The Conant--Manathunga family via trace calculus}\label{app:burau}
 
This appendix computes, from scratch, the two-variable Alexander polynomials of the links $L_2=K\cup C_2$ and $L_1=K_1\cup C_2$ for the order-$4$ braid family of \cite{CM}, verifying the formulas displayed in \S\ref{subsec:fibonacci} and, with them, every level of the machinery of \S\S\ref{sec:levelwise}--\ref{sec:telescope} on an infinite family. The family consists of the closures of the $3$-braids
\[
\beta=\alpha^2,\qquad \alpha=\sigma_1^{\,n}\sigma_2^{\,m}\sigma_1^{-n}\sigma_2^{-m}\in B_3,\qquad n,m\ \text{odd},
\]
braided about the axis $C_2$, the order-$4$ rotary reflection acting by a quarter-turn of the pages; the $\pi$-rotation $r=g^2$ advances pages by a half-turn, so the quotient knot $K_1$ is the closure of the half-word $\alpha$ about the same axis. (The underlying permutations of $\alpha$ and $\alpha^2$ are the two $3$-cycles, so both closures are knots, and $\ell=3$ at both levels.)
 
\subsection{Burau computes the axis polynomial}\label{app:buraugeneral}
Let $\widehat\beta$ be the closure of $\beta\in B_k$ with $\widehat\beta$ a knot, and $C_2$ the braid axis. The exterior $V$ of $\widehat\beta\cup C_2$ fibers over $S^1$ with fiber the $k$-punctured disk $D_k$ and monodromy $h_\beta$; on $H_1(V)=\langle t,x\rangle$ the strand meridians all represent $t$ (the closure being connected) and the meridian of $C_2$ is the base circle of the fibration. Unwrapping first $x$ and then $t$ exhibits the universal abelian cover as the mapping torus, over $\Lambda=\Z[t^{\pm1}]$, of a lift $\tilde h_\beta$ acting on the infinite cyclic cover $\widetilde D_k$. The torsion of an algebraic mapping torus is $\prod_i\det\big(x\,\tilde h_*-\mathrm{id}\mid H_i(\widetilde D_k)\otimes\Q(t)\big)^{(-1)^{i+1}}$; here $H_0(\widetilde D_k)=\Lambda/(t-1)$ dies over $\Q(t)$, and $H_1(\widetilde D_k)\cong\Lambda^{k-1}$ carries, by definition, the reduced Burau representation $\bar\psi\colon B_k\to GL_{k-1}(\Lambda)$. Hence, by the identification $\tau\eqdot\Delta$ of \S\ref{sec:conventions},
\begin{equation}\label{eq:buraupoly}
\Delta_{\widehat\beta\cup C_2}(t,x)\;\eqdot\;\det\big(x\,\bar\psi_\beta(t)-I_{k-1}\big),
\end{equation}
the choice of lift $\tilde h_\beta$ (ambiguous by $t^{\Z}$) effecting $x\mapsto xt^k$, within the $\eqdot$ ambiguity. This is essentially Burau's theorem \cite{Bur}; see \cite{Mor} for the extension to arbitrary diagrams. Since $\det\bar\psi(\sigma_i)=-t$, the extreme $x$-coefficients of \eqref{eq:buraupoly} are the units $\pm t^{\Z}$. Setting $x=1$ in \eqref{eq:buraupoly} and comparing with \eqref{eq:torres} recovers the classical $\det(\bar\psi_\beta-I)\eqdot(1+t+\cdots+t^{k-1})\,\Delta_{\widehat\beta}(t)$.
 
\subsection{Normalization and traces}
Work in $B_3$ with the reduced Burau matrices
\[
A_1=\bar\psi(\sigma_1)=\begin{pmatrix}-t&1\\0&1\end{pmatrix},\qquad
A_2=\bar\psi(\sigma_2)=\begin{pmatrix}1&0\\t&-t\end{pmatrix},\qquad \det A_1=\det A_2=-t.
\]
(Any of the standard conventions, differing by conjugation and transposition, gives the same traces below.) Let $s=t^{1/2}$, $z=s-s^{-1}$, and normalize $a_1:=A_1/(is)$, $a_2:=A_2/(is)$, so $a_1,a_2\in SL_2$. Because $\alpha$ is a commutator, the scalars cancel identically:
\[
N_\alpha:=\bar\psi(\alpha)=A_1^nA_2^mA_1^{-n}A_2^{-m}=a_1^na_2^ma_1^{-n}a_2^{-m},\qquad \det N_\alpha=1 ,
\]
so $N_\alpha\in SL_2\big(\Z[t^{\pm1}]\big)$ on the nose and no unit-fixing is needed. The three basic traces are
\[
\operatorname{tr}a_1=\operatorname{tr}a_2=\frac{1-t}{is}=iz,\qquad \operatorname{tr}(a_1a_2)=\frac{\operatorname{tr}(A_1A_2)}{(is)^2}=\frac{-t}{-t}=1 .
\]
 
Let $F_k$ denote the Fibonacci polynomials ($F_0=0$, $F_1=1$, $F_{k+1}=zF_k+F_{k-1}$) and $L_k:=F_{k+1}+F_{k-1}$ the Lucas polynomials, and write $D:=z^2+4$. We use the standard identities, valid for all indices and provable from the Binet-type formulas $F_k=(\sigma^k-\tau^k)/(\sigma-\tau)$, $L_k=\sigma^k+\tau^k$ with $\sigma\tau=-1$, $\sigma+\tau=z$, and noting that $D=(\sigma-\tau)^2$:
\begin{equation}\label{eq:lucasids}
\begin{split}
L_k^2=D\,F_k^2+4(-1)^k,\qquad
L_jL_k=L_{j+k}+(-1)^kL_{j-k},\\
D\,F_jF_k=L_{j+k}-(-1)^kL_{j-k},\qquad
L_{k+2}+L_{k-2}=(z^2+2)L_k .
\end{split}
\end{equation}
 
\begin{lemma}\label{lem:powers}
For the Chebyshev-type polynomials $S_0=1$, $S_1=x$, $S_{k+1}=x S_k-S_{k-1}$ one has $S_{k}(iz)=i^{\,k}F_{k+1}(z)$. Consequently, for $c\in SL_2$ with $\operatorname{tr}c=iz$,
\[
c^{\,k}=i^{\,k-1}F_k\,c-i^{\,k-2}F_{k-1}\,I,\qquad \operatorname{tr}(c^{\,k})=i^{\,k}L_k .
\]
\end{lemma}
 
\begin{proof}
The first claim is induction: $i^{k+1}F_{k+2}=iz\cdot i^kF_{k+1}-i^{k-1}F_k$ is the Fibonacci recursion. The power formula follows from Cayley--Hamilton $c^2=(\operatorname{tr}c)\,c-I$, allowing any power of $c$ to be rewritten as a linear combination of $c$ and $I$. These coefficients satisfy the defining recursion for $c$, so $c^k=S_{k-1}(\operatorname{tr}c)\,c-S_{k-2}(\operatorname{tr}c)I$, and taking traces gives $\operatorname{tr}(c^k)=izS_{k-1}-2S_{k-2}=S_k-S_{k-2}=i^k(F_{k+1}+F_{k-1})$.
\end{proof}
 
\begin{lemma}\label{lem:producttrace}
For $m$ odd, $\operatorname{tr}(a_1^na_2^m)=i^{\,n+m-2}H$ with
\[
H\;=\;F_nF_m-z\big(F_nF_{m-1}+F_{n-1}F_m\big)-2F_{n-1}F_{m-1}
\;=\;\frac{(z^2+3)\,L_{n-m}-L_{n+m}}{D},
\]
and moreover
\begin{equation}\label{eq:keytrace}
2H\;=\;(D-2)\,F_nF_m\;-\;L_nL_m .
\end{equation}
\end{lemma}
 
\begin{proof}
Expanding by Lemma~\ref{lem:powers} and using $\operatorname{tr}(a_1a_2)=1$, $\operatorname{tr}a_1=\operatorname{tr}a_2=iz$, $\operatorname{tr}I=2$:
\[
\operatorname{tr}(a_1^na_2^m)=i^{n+m-2}\Big[F_nF_m\cdot1-zF_nF_{m-1}-zF_{n-1}F_m+i^{-2}\,2F_{n-1}F_{m-1}\Big],
\]
which is the first expression. For the closed form, convert products to sums by \eqref{eq:lucasids} (with $m$ odd, so $(-1)^m=-1$, $(-1)^{m-1}=1$) and eliminate $z$ via $zL_k=L_{k+1}-L_{k-1}$:
\begin{align*}
D\,H&=\big(L_{n+m}+L_{n-m}\big)-z\big(2L_{n+m-1}-L_{n-m+1}+L_{n-m-1}\big)-2\big(L_{n+m-2}-L_{n-m}\big)\\
&=L_{n+m}+L_{n-m}-2\big(L_{n+m}-L_{n+m-2}\big)+\big(L_{n-m+2}-L_{n-m}\big)\\
&\quad-\big(L_{n-m}-L_{n-m-2}\big)-2L_{n+m-2}+2L_{n-m}\\
&=-L_{n+m}+L_{n-m}+\big(L_{n-m+2}+L_{n-m-2}\big)
\;=\;-L_{n+m}+(z^2+3)L_{n-m},
\end{align*}
by the last identity of \eqref{eq:lucasids}. Finally, again by \eqref{eq:lucasids} with $m$ odd,
\[
\begin{split}
(D-2)F_nF_m-L_nL_m=\frac{(z^2+2)\big(L_{n+m}+L_{n-m}\big)}{D}-\big(L_{n+m}-L_{n-m}\big)\\
=\frac{-2L_{n+m}+(2z^2+6)L_{n-m}}{D}=2H. \qedhere
\end{split}
\]
\end{proof}
 
\subsection{The commutator trace}
 
\begin{proposition}\label{prop:trace}
For $n,m$ odd, $\operatorname{tr}N_\alpha=2-(z^2+3)\,F_n^2F_m^2=2-T$, in the notation $T=(w+3)F_n^2F_m^2$ of \S\ref{subsec:fibonacci}.
\end{proposition}
 
\begin{proof}
By the Fricke identity, for $X,Y\in SL_2$ with $x=\operatorname{tr}X$, $y=\operatorname{tr}Y$, $v=\operatorname{tr}(XY)$:
\[
\operatorname{tr}\big(XYX^{-1}Y^{-1}\big)-2\;=\;x^2+y^2+v^2-xyv-4\;=:\;\kappa .
\]
Take $X=a_1^n$, $Y=a_2^m$. By Lemma~\ref{lem:powers} and \eqref{eq:lucasids} with $n,m$ odd: $x^2=(i^nL_n)^2=-L_n^2=4-DF_n^2$, likewise $y^2=4-DF_m^2$; by Lemma~\ref{lem:producttrace}, $v=i^{n+m-2}H$ with $n+m$ even, so $v^2=H^2$ and $xyv=i^{n}i^{m}i^{n+m-2}L_nL_mH=-L_nL_mH$. Hence
\[
\kappa=4-D\big(F_n^2+F_m^2\big)+H\big(H+L_nL_m\big).
\]
By \eqref{eq:keytrace}, the last summand factors as a difference of squares:
\[
H\big(H+L_nL_m\big)=\frac{(D-2)F_nF_m-L_nL_m}{2}\cdot\frac{(D-2)F_nF_m+L_nL_m}{2}
=\frac{(D-2)^2F_n^2F_m^2-L_n^2L_m^2}{4},
\]
and substituting $L_n^2L_m^2=(DF_n^2-4)(DF_m^2-4)$ from \eqref{eq:lucasids}:
\[
\begin{split}
H\big(H+L_nL_m\big)=\frac{(D^2-4D+4)-D^2}{4}F_n^2F_m^2+D\big(F_n^2+F_m^2\big)-4\\
=-(D-1)F_n^2F_m^2+D\big(F_n^2+F_m^2\big)-4 .
\end{split}
\]
Everything cancels except $\kappa=-(D-1)F_n^2F_m^2=-(z^2+3)F_n^2F_m^2$, and $\operatorname{tr}N_\alpha=\kappa+2$.
\end{proof}
 
\subsection{The two-variable polynomials}
 
\begin{theorem}\label{thm:appendixmain}
For the family above, with $T=(w+3)F_n^2F_m^2$:
\begin{align}
\Delta_{L_1}(t,y)&\;\eqdot\;\det\big(yN_\alpha-I\big)\;=\;y^2+(T-2)\,y+1,\label{eq:level1}\\
\Delta_{L_2}(t,x)&\;\eqdot\;\det\big(xN_\alpha^2-I\big)\;=\;x^2-\big((T-2)^2-2\big)\,x+1.\label{eq:level2}
\end{align}
In particular: \textup{(i)} the covering factorization \textup{(}Proposition~\ref{prop:covering}\textup{)} holds by the matrix identity $(yN_\alpha-I)(yN_\alpha+I)=y^2N_\alpha^2-I$; \textup{(ii)} both polynomials are visibly palindromic in the axis variable \textup{(}Theorem~\ref{thm:normalform}\textup{)}, by unimodularity of $N_\alpha$; \textup{(iii)} Torres at level $1$ gives $\Delta_{L_1}(t,1)=T=\qint3\cdot F_n^2F_m^2$, so $\Conway_{K_1}=F_n^2F_m^2$, the Hartley--Kawauchi square of the quotient knot \textup{(}Lemma~\ref{lem:HKsquare}\textup{)}, with $Q_1(u,v)=v+(T-2)$, $R_1=Q_1(u,-2)=T-4$, $R_1(0)=-1=\eps_3$, $\widehat R_1=4-T$; \textup{(iv)} Torres at level $2$ gives
\[
\Delta_{L_2}(t,1)=4-(T-2)^2=T\,(4-T)=\qint3\cdot F_n^2F_m^2\big(4-(w+3)F_n^2F_m^2\big),
\]
recovering the formula of \cite{CM} for $\Conway_K$ and instantiating Theorem~\ref{thm:telescope} as $\Conway_K=\Conway_{K_1}\cdot\widehat R_1$.
\end{theorem}
 
\begin{proof}
Since $N_\alpha\in SL_2$, $\det(yN_\alpha-I)=y^2\det N_\alpha-y\operatorname{tr}N_\alpha+1=y^2-(2-T)y+1$ by Proposition~\ref{prop:trace}, and $\operatorname{tr}(N_\alpha^2)=(\operatorname{tr}N_\alpha)^2-2=(T-2)^2-2$ with $\det N_\alpha^2=1$. The identification with the Alexander polynomials is \eqref{eq:buraupoly}, applied to $\alpha$ at level $1$ and $\beta=\alpha^2$ at level $2$; note that the level-$2$ monodromy is the square of the level-$1$ monodromy precisely because the $\pi$-rotation about the axis carries the page at angle $\varphi$ to the page at $\varphi+\pi$, identifying $K_1$ with the closure of the half-word. The listed consequences are immediate computations, matching \S\ref{subsec:fibonacci} line by line.
\end{proof}
 
\begin{remark}
The computations in this section were heavily aided by Claude, though all carefully checked.
Rather than directly computing the elements of the products of matrices appearing in the Burau representation, as was done in \cite{CM}, 
the trace calculus applied to $\big(\operatorname{tr}a_1,\operatorname{tr}a_2,\operatorname{tr}(a_1a_2)\big)$ is used instead.
\end{remark}

\section{Five-strand examples}\label{app:fivestrand}

The three-strand computation of the previous section is structurally incapable of
producing an interesting $E$: there the monodromy is a $2\times2$ unimodular matrix,
$Q_1=v-\operatorname{tr}\bar\psi$ has a single free coefficient, and
$E=\tfrac14\big(Q_1(2)-Q_1(-2)\big)=1$ identically, whatever the braid word. In this
section we work one rung up the ladder. Note first that if $\widehat\beta$ is
braided about the axis then $\ell=\operatorname{lk}(\widehat\beta,C_2)$ equals the number
of strands, so an equivariant braid position for an $\ell$-rotary knot lives in
$B_\ell$; the first case with room for a nontrivial $E$ is therefore $\ell=5$, where the
reduced Burau matrices are $4\times4$. 

Similar computations to the previous section yield
Table~\ref{table:fivestrand}, which lists the invariants for some $5$ strand examples. Recall for a braid $\gamma$ that $\bar \gamma$ is the reflection through the plane of the paper, i.e. switches all crossings. Also note that the permutation associated with $\gamma$ must be a 5 cycle for the closure to be a knot. These computations were performed by Claude using several supplementary lemmas not included here, and have not been independently verified.

\begin{table}[ht]\small
\renewcommand{\arraystretch}{1.35}
\begin{tabular}{@{}llll@{}}
\hline
$\gamma$ & $E$ & $M$ & $\Conway_{K_1}$\\
\hline
$\sigma_1\sigma_2\sigma_3\sigma_4$ & $w+1$ & $w^2+3w+3$ & $1$\\
$\sigma_1\sigma_2^{-1}\sigma_3\sigma_4$ & $2w+1$ & $w^2+w+3$ & $1$\\
$\sigma_1^{3}\sigma_2\sigma_3\sigma_4$ & $2w+1$ & $w^4+7w^3+16w^2+11w+3$ & $(w+1)^2$\\
$\sigma_4\sigma_2^{-1}\sigma_3^{-2}\sigma_1^{-1}\sigma_3^{-1}$ & $w^3+5w^2+7w+1$ & $w^4+5w^3+6w^2+w+3$ & $(w+1)^2$\\
$\sigma_2^{-1}\sigma_4^{-1}\sigma_1^{-1}\sigma_3^{-1}\sigma_2\sigma_3^{-1}$ & $w^3-w^2-w+1$ & $w^4+w^3-2w^2-3w+3$ & $(w-1)^2$\\
\hline
\end{tabular}
\medskip
\caption{Five-strand examples $\alpha=\gamma\bar\gamma$, $K=\widehat{\alpha^2}$. In each
row $\qint5\,\Conway_K=M^2-4E^2$; the resulting Conway polynomials are, in order,
$w^2+w+1$;\; $w^2-3w+1$;\; $(w+1)^2(w^4+7w^3+16w^2+7w+1)$;\;
$(w+1)^2(w^4+3w^3-4w^2-13w+1)$;\; $(w-1)^4(w^2+w+1)$.}
\label{table:fivestrand}
\end{table}


\begin{thebibliography}{ABC}

\bibitem[Agl]{Agle} K.~Agle, \emph{Alexander and Conway polynomials of Torus knots}, Master’s Thesis, University of
Tennessee, 2012. \url{http://trace.tennessee.edu/utk_gradthes/1127}

\bibitem[Bur]{Bur} W.~Burau, \emph{\"Uber Zopfgruppen und gleichsinnig verdrillte Verkettungen}, Abh. Math. Sem. Univ. Hamburg \textbf{11} (1936), 179--186.

\bibitem[BC]{BC} K.~Boyle and W.~Chen, \emph{Negative amphichiral knots and the half-Conway polynomial}, Rev. Mat. Iberoam. 40, No. 2, 581--622 (2024). 

\bibitem[Con]{Con} J.~Conant, \emph{Chirality and the Conway polynomial}, Top. Proc. 30, no 1 (2006) pp. 153--162.

\bibitem[CM]{CM} J.~Conant and V.~Manathunga, \emph{The Conway polynomial and amphicheiral knots}, J. Knot Theory Ramifications \textbf{26} (2017), no.~5, Article ID 1750027.

\bibitem[DL]{DL} J.~Dinkelbach and B.~Leeb, \emph{Equivariant Ricci flow with surgery and applications to finite group actions on geometric 3-manifolds}, Geom. Topol. \textbf{13} (2009), 1129--1173.

\bibitem[FV]{FV} S.~Friedl and S.~Vidussi, \emph{A survey of twisted Alexander polynomials}, in: The Mathematics of Knots, Contrib.\ Math.\ Comput.\ Sci.\ 1, Springer, Heidelberg, 2011, pp.~45--94. 

\bibitem[EHW]{EHW} N.~Ermotti, C.~V.~Quach Hongler and C.~Weber, \emph{On the Kawauchi conjecture about the Conway polynomial of achiral knots}, J. Knot Theory Ramifications 21, No. 9, Article ID 1250092, 9 p. (2012). 

\bibitem[H1]{H1} R.~Hartley, \emph{Invertible amphicheiral knots}, Math. Ann. \textbf{252} (1979/80), 103--109. 

\bibitem[H2]{H2} R.~Hartley, \emph{Knots with free period}, Canad. J. Math. \textbf{33} (1981), 91--102. 

\bibitem[HK]{HK} R.~Hartley and A.~Kawauchi, \emph{Polynomials of amphicheiral knots}, Math. Ann. \textbf{243} (1979), 63--70.

\bibitem[HKL]{HKL} C.~Herald, P. Kirk, C. Livingston, \emph{
Metabelian representations, twisted Alexander polynomials, knot slicing, and mutation.}
Math. Z. 265, No. 4, 925--949 (2010). 

\bibitem[Kaw]{Kaw} A.~Kawauchi, \emph{H-cobordism. I. The groups among three dimensional homology handles.} Osaka J. Math. 13 (1976), no. 3, 567--590.

\bibitem[Mil]{Mil} J.~Milnor, \emph{A duality theorem for Reidemeister torsion}, Ann. of Math. (2) \textbf{76} (1962), 137--147.

\bibitem[Mor]{Mor} H.~R.~Morton, \emph{Threading knot diagrams}, Math. Proc. Cambridge Philos. Soc. \textbf{99} (1986), 247--260.

\bibitem[Tor]{Tor} G.~Torres, \emph{On the Alexander polynomial}, Ann. of Math. (2) \textbf{57} (1953), 57--89.

\bibitem[Tur]{Tur} V.~Turaev, \emph{Introduction to Combinatorial Torsion}, Lectures in Mathematics ETH Z\"urich, Birkh\"auser, 2001.

\end{thebibliography}
\end{document}